\documentclass[reqno,a4paper,11pt]{amsart}
\usepackage[all,poly]{xy}
\usepackage{amsfonts}
\usepackage{amssymb}
\usepackage{amsmath}
\usepackage{color}
\usepackage[pagebackref,colorlinks]{hyperref}
\usepackage{enumerate}
\usepackage{hhline}
\usepackage{tikz}
\usetikzlibrary{calc}
\makeatletter
\tikzset{
    @pos/.style={@pos1={#1},@pos2={#1}},
    @ratio/.style={@ratio1={#1},@ratio2={#1}},
    @delta/.style={@delta1={#1},@delta2={#1}},
    @edge/.style={@@edge/.append style={#1}},
    @edge 0/.style={@@edge 0/.append style={#1}},
    @edge 1/.style={@@edge 1/.append style={#1}},
    @edge 2/.style={@@edge 2/.append style={#1}},
    @edge 3/.style={@@edge 3/.append style={#1}},
    @edge 4/.style={@@edge 4/.append style={#1}},
    @pos1/.store in=\qrr@posA,
    @pos2/.store in=\qrr@posB,
    @ratio1/.store in=\qrr@ratioA,
    @ratio2/.store in=\qrr@ratioB,
    @delta1/.store in=\qrr@deltaA,
    @delta2/.store in=\qrr@deltaB,
    @pos=.5,
    @ratio=.5,
    @delta=.1,
}
\newcommand*{\connectThree}[4][]{
    \begingroup
    \tikzset{#1}
    \coordinate (@aux1) at ($(#2)!\qrr@ratioA!(#3)$);
    \coordinate (@aux2) at ($(#4)!\qrr@posA!(@aux1)$);
    \path (@aux2) edge[@@edge/.try, @@edge 0/.try, @@edge 3/.try] (#4);
    \draw[@@edge/.try, @@edge 1/.try] (@aux2) .. controls ($(#4)!\qrr@posA+\qrr@deltaA!(@aux1)$) .. (#2);
    \draw[@@edge/.try, @@edge 2/.try] (@aux2) .. controls ($(#4)!\qrr@posA+\qrr@deltaA!(@aux1)$) .. (#3);
    \endgroup
}
\renewcommand*{\connectThree}[4][]{\connectFour[#1, @@edge 4/.style={draw=none}, @ratio2=0]{#2}{#3}{#4}{0,0}}

\newcommand*{\connectFour}[6][]{
    \begingroup
    \tikzset{#1}
    \coordinate (@aux1a) at ($(#2)!\qrr@ratioA!(#3)$);
    \coordinate (@aux1b) at ($(#4)!\qrr@ratioB!(#5)$);
    \coordinate (@aux2a) at ($(@aux1b)!\qrr@posA!(@aux1a)$);
    \coordinate (@aux2b) at ($(@aux1a)!\qrr@posB!(@aux1b)$);
    \path (@aux2a) edge[@@edge/.try,@@edge 0/.try] (@aux2b);
    \draw[@@edge/.try,@@edge 1/.try] (@aux2a) .. controls ($(@aux1b)!\qrr@posA+\qrr@deltaA!(@aux1a)$) .. (#2);
    \draw[@@edge/.try,@@edge 2/.try] (@aux2a) .. controls ($(@aux1b)!\qrr@posA+\qrr@deltaA!(@aux1a)$) .. (#3);
    \draw[@@edge/.try,@@edge 3/.try] (@aux2b) .. controls ($(@aux1a)!\qrr@posB+\qrr@deltaB!(@aux1b)$) .. (#4);
    \draw[@@edge/.try,@@edge 4/.try] (@aux2b) .. controls ($(@aux1a)!\qrr@posB+\qrr@deltaB!(@aux1b)$) .. (#5);
    \draw[draw=none] (@aux1a) -- (@aux1b) node[midway,above,sloped,font=\tiny,shape=rectangle,inner xsep=+0pt,draw=none,align=center,fill=white,fill opacity=.75,outer ysep=\pgflinewidth,text opacity=1] {#6};
    \endgroup
}
\makeatother
     
\theoremstyle{plain}
\newtheorem {lemma}{Lemma}
\newtheorem {proposition}[lemma]{Proposition}
\newtheorem {theorem}[lemma]{Theorem}
\newtheorem {corollary}[lemma]{Corollary}

\newtheorem {replacement lemma}[lemma]{Replacement Lemma}

\theoremstyle{definition}
\newtheorem{definition}[lemma]{Definition}

\newtheorem{remark}[lemma]{Remark}
\newtheorem {example}[lemma]{Example}

\newcommand{\R}{\mathbb{R}}
\newcommand{\N}{\mathbb{N}}
\newcommand{\Z}{\mathbb{Z}}
\newcommand{\X}{\langle X \rangle}
\def\Oo{\mathcal{O}}
\newcommand{\GKdim}{\operatorname{GKdim}}
\newcommand{\M}{\operatorname{\mathcal{M}^{ab}}}
\newcommand{\A}{\operatorname{\mathcal{A}_K^{loc}}}

\newcommand{\Reg}{\operatorname{Reg}}
\newcommand{\Mult}{\operatorname{Mult}}
\newcommand{\supp}{\operatorname{supp}}
\newcommand{\Hy}{\operatorname{\mathcal H}}
\newcommand{\nod}{\operatorname{nod}}
\newcommand{\NF}{\operatorname{NF}}
\newcommand{\irr}{\operatorname{irr}}
\newcommand{\gr}{\operatorname{gr}}
\newcommand{\Gr}{\operatorname{Gr}}
\newcommand{\Mod}{\operatorname{Mod}}
\newcommand{\proj}{\operatorname{proj}}
\newcommand{\Id}{\operatorname{Id}}
\newcommand{\ann}{\operatorname{ann}}
\newcommand{\Span}{\operatorname{span}}

\title{Leavitt path algebras of hypergraphs}

\author{Raimund Preusser}
\address{Department of Mathematics,
University of Brasilia, Brazil}
\email{raimund.preusser@gmx.de}


\begin{document}

\begin{abstract} 
We define Leavitt path algebras of hypergraphs generalizing simultaneously Leavitt path algebras of finitely separated graphs and Leavitt path algebras of row-finite vertex-weighted graphs. We find linear bases for those algebras, compute their Gelfand-Kirillov dimension, obtain some results on ring-theoretic properties like simplicity, von Neumann regularity and Noetherianess and investigate their $K$-theory and graded $K$-theory. By doing so we obtain new results on the Gelfand-Kirillov dimension and graded $K$-theory of Leavitt path algebras of separated graphs and on the graded $K$-theory of weighted Leavitt path algebras.
\end{abstract}

\maketitle
\tableofcontents
\section{Introduction}
In a series of papers \cite{vitt56, vitt57, vitt62, vitt65} William Leavitt studied algebras that are now denoted by $L_K(n,n+k)$ and have been coined Leavitt algebras. Let $X=(x_{ij})$ and $Y=(y_{ji})$ be $(n+k)\times n$ and $n\times (n+k)$ matrices consisting of symbols $x_{ij}$ and $y_{ji}$, respectively. Then for a field $K$, $L_K(n,n+k)$ is the associative, unital $K$-algebra generated by all $x_{ij}$ and $y_{ji}$ subject to the relations $XY=I_{n+k}$ and $YX=I_n$. Leavitt established that the algebra $L_K(n,n+k)$ has module type $(n,k)$. Furthermore he showed that the algebras $L_K(1,k+1)$ are simple and the algebras $L_K(n,n+k)$, $n\geq 2$ domains. Recall that a ring $R$ has module type $(n,k)$ if $n$ and $k$ are the least positive integers such that $R^n\cong R^{n+k}$ as left $R$-modules. 

Leavitt path algebras were introduced by G. Abrams and G. Aranda Pino in 2005 \cite{aap05} and independently by P. Ara, M. Moreno and E. Pardo in 2007 \cite{Ara_Moreno_Pardo} as $K$-algebras associated to directed graphs. For the directed graph with one vertex and $k+1$ loops one recovers the Leavitt algebra $L_K(1,k+1)$. The definition and the development of the theory were inspired on the one hand by Leavitt's construction of $L_K(1,k+1)$ and on the other hand by the Cuntz algebras $\Oo_n$ \cite{cuntz1} and the Cuntz-Krieger algebras in $C^*$-algebra theory \cite{raeburn}. The Cuntz algebras and later Cuntz-Krieger type $C^*$-algebras revolutionised $C^*$-theory, leading ultimately to the astounding
Kirchberg-Phillips classification theorem~\cite{phillips}. The Leavitt path algebras have created the same type of stir in the algebraic community. The development of Leavitt path algebras and its interaction with graph $C^*$-algebras have been well-documented in several publications and we refer the reader to~\cite{abrams-ara-molina} and the references therein. 
 
Since their introductions, there have been several attempts to introduce a generalisation of the Leavitt path algebras which would cover the algebras $L_K(n,n+k)$ for any $n\geq 1$, as well. P. Ara and K. Goodearl introduced Leavitt path algebras of separated graphs in 2012~\cite{aragoodearl}. They showed that any Leavitt algebra $L_K(n,n+k)$ is a corner ring of a Leavitt path algebra of a finitely separated graph. Weighted Leavitt path algebras were introduced by R. Hazrat in 2013 \cite{hazrat13}. For the weighted graph with one vertex and $n+k$ loops of weight $n$ one recovers the Leavitt algebra $L_K(n,n+k)$. If the weights of all the edges are $1$, then the weighted Leavitt path algebras reduce to the usual Leavitt path algebras. 

In this paper we define and investigate Leavitt path algebras of hypergraphs. A (directed) hypergraph can be thought of as a directed graph where the edges are allowed to have multiple sources and ranges. Usually the source and the range of an edge in a hypergraph (a ``hyperedge") are required to be nonempty, finite subsets of the vertex set. In this paper we use a slightly modified definition of a hypergraph. We only require that the source and the range of a hyperedge are multisets over the vertex set with finite, nonempty support. Hence a vertex may appear in the source (resp. range) of a hyperedge more than once. 

The rest of the paper is organised as follows.

In Section 2 we recall some standard notation which used throughout the paper.

In Section 3 we define the Leavitt path algebra $L_K(H)$ of a hypergraph $H$. We show that these algebras generalise the Leavitt path algebras of finitely separated graphs and row-finite vertex-weighted graphs. Moreover, we prove that a Leavitt path algebra $L_K(H)$ of a hypergraph $H$ is an involutary $\Z^n$-graded algebra with local units where $n$ is the supremum over all the cardinalities of sources of hyperedges in $H$ (the cardinality $|M|$ of a multiset $M$ with finite, nonempty support is defined in the first paragraph of Section 3).

In Section 4 we show that the category $\Hy$ of hypergraphs admits arbitrary direct limits. The main result of the section is Proposition \ref{proplim2} which says that for any hypergraph $H$ its Leavitt path algebra $L_K(H)$ is a direct limit of a direct system of Leavitt path algebras of finite hypergraphs.

In Section 5 we find linear bases for the Leavitt path algebras of hypergraphs. These bases play an important role in Sections 6, 7 and 8.

In Section 6 we compute the Gelfand-Kirillov dimension of a Leavitt path algebra of a hypergraph. The GK dimension of a Leavitt path algebra of a row-finite weighted graph was already known, see \cite{preusser}. But the result for the GK dimension of a Leavitt path algebra of a finitely separated graph seems to be new.

In Section 7 we show that the Leavitt path algebra $L_K(H)$ has a ``local valuation'' provided the hypergraph $H$ satisfies Condition (LV), i.e. for any hyperedge $h$ the cardinalities of $s(h)$ and $r(h)$ are at least $2$. We deduce that if $H$ is a connected hypergraph which satisfies Condition (LV) and has at least one hyperedge, then $L_K(H)$ is prime, nonsingular, not von Neumann regular and semiprimitive. Furthermore we classify the hypergraphs $H$ such that $L_K(H)$ is a domain.

In Section 8 we show that a Leavitt path algebra $L_K(H)$ that is finite-dimensional as a $K$-vector space or simple or left Artinian or right Artininan or von Neumann regular is isomorphic to a Leavitt path algebra of a finitely separated graph. We also prove a result on the Noetherianess of $L_K(H)$.

In Section 9 we compute the monoid $V(L_K(H))$ for any hypergraph $H$. Recall that for a ring $R$ with local units, $V(R)$ is the set of all isomorphism classes of unital finitely generated projective left $R$-modules which becomes an abelian monoid by defining $[P]+[Q]:=[P\oplus Q]$. The Grothendieck group $K_0(R)$ is the group completion of $V(R)$.

In Section 10 we compute the monoid $V^{\gr}(L_K(H))$ with respect to any grading on $L_K(H)$ induced by an ``admissible weight map" (we show that for example the standard $\Z^n$-grading on $L_K(H)$ defined in Section 3 is induced by an admissible weight map). To the best knowledge of the author, the graded $V$-monoid of a Leavitt path algebra of a finitely separated graph or a row-finite vertex-weighted graph had not been computed before. Recall that for a graded ring $R$ with graded local units, $V^{\gr}(R)$ is the set of all isomorphism classes of graded unital finitely generated projective left $R$-modules which becomes an abelian monoid by defining $[P]+[Q]:=[P\oplus Q]$. The graded Grothendieck group $K^{gr}_0(R)$ is the group completion of $V^{\gr}(R)$. 

\section{Notation}
Throughout the paper $K$ denotes a field. By a ring resp. $K$-algebra we mean an associative (but not necessarily commutative or unital) ring resp. $K$-algebra. By an ideal of a ring we mean a twosided ideal. $\N$ denotes the set of positive integers, $\N_0$ the set of nonnegative integers, $\Z$ the set of integers and $\R_+$ the set of positive real numbers. 


\section{Leavitt path algebras of hypergraphs}
Recall that a multiset is a modification of the concept of a set that, unlike a set, allows for multiple instances for each of its elements. Formally a {\it multiset} $M$ over a set $X$ is a function $M:X\rightarrow \N_0$. If $x\in X$, then $M(x)$ is called the {\it multiplicity} of $x$. The set $M_{\supp}:=\{x\in X \mid M(x)>0\}$ is called the {\it support} of $M$. We denote by $\Mult(X)$ the set of all multisets over $X$ that have finite, nonempty support. For an element $M\in\Mult(X)$ we set $|M|:=\sum\limits_{x\in M_{\supp}} M(x)$.  
 
\begin{definition}
A {\it (directed) hypergraph} is a quadruple $H=(H^0,H^1,s,r)$ where $H^0$ and $H^1$ are sets and $s,r:H^1\rightarrow \Mult(H^0)$ are maps. The elements of $H^0$ are called {\it vertices} and the elements of $H^1$ {\it hyperedges}. $H$ is called {\it finite} if $H^0$ and $H^1$ are finite sets and {\it empty} if $H^0=H^1=\emptyset$. In this article all hypergraphs are assumed to be nonempty.
\end{definition}
\begin{definition}\label{defhlpa}
Let $H$ be a hypergraph. For any $h\in H^1$ write $s(h)=\{s(h)_1,\dots,s(h)_{|s(h)|}\}$ and $r(h)=\{r(h)_1,\dots,r(h)_{|r(h)|}\}$. The $K$-algebra presented by the generating set 
\[\{v,h_{ij},h_{ij}^*\mid v\in H^0, h\in H^1, 1\leq i\leq |s(h)|,  1\leq j\leq |r(h)|\}\]
and the relations
\begin{enumerate}[(i)]
\item $uv=\delta_{uv}u\quad(u,v\in H^0)$,
\medskip
\item $s(h)_ih_{ij}=h_{ij}=h_{ij}r(h)_j,~r(h)_jh_{ij}^*=h_{ij}^*=h_{ij}^*s(h)_i\quad(h\in H^1, 1\leq i\leq |s(h)|,  1\leq j\leq |r(h)|)$,
\medskip
\item $\sum\limits_{k=1}^{|r(h)|}h_{ik}h_{jk}^*= \delta_{ij}s(h)_i\quad(h\in H^1, ~1\leq i,j\leq |s(h)|)$ and
\medskip
\item $\sum\limits_{k=1}^{|s(h)|}h_{ki}^*h_{kj}= \delta_{ij}r(h)_i\quad(h\in H^1,~1\leq i,j\leq |r(h)|)$
\end{enumerate}
is called the {\it Leavitt path algebra of $H$} and is denoted by $L_K(H)$. 
\end{definition}

\begin{remark}\label{remhlpa}
$~$\\
\vspace{-0.6cm}
\begin{enumerate}[(a)]
\item One checks easily that the isomorphism class of the algebra $L_K(H)$ does not depend on the chosen ordering of the elements of the multisets $s(h)$ and $r(h)$ ($h\in H^1$).
\item 
Let $H$ be a hypergraph. Define a directed graph $E=(E^0,E^1, s',r')$ by $E^0=H^0$, $E^1=\{h_{ij} \mid h\in H^1, 1\leq i\leq |s(h)|, 1\leq j\leq |r(h)|\}$, $s'(h_{ij})=s(h)_i$ and $r'(h_{ij})=r(h)_j$. The graph $E$ is called the {\it directed graph associated to $H$}. Let $\hat E$ be the double graph of $E$ and $K\hat E$ the path $K$-algebra of $\hat E$ (see for example \cite[Remark 1.2.4]{abrams-ara-molina}). Then $L_K(H)$ is isomorphic to the quotient of $K\hat E$ by the ideal of $K\hat E$ generated by relations (iii) and (iv) in Definition \ref{defhlpa}.
\item Let $H$ be a hypergraph and $A$ a $K$-algebra which contains a set $X:=\{a_v, b_{h,i,j}, c_{h,i,j}\mid v\in H^0, h\in H^1, 1\leq i\leq |s(h)|,  1\leq j\leq |r(h)|\}$ such that\\
\vspace{-0.1cm}
\begin{enumerate}[(i)]
\item the $a_v$'s are pairwise orthogonal idempotents, 
\medskip
\item
$a_{s(h)_i}b_{h,i,j}=b_{h,i,j}=b_{h,i,j}a_{r(h)_j},~a_{r(h)_j}c_{h,i,j}=c_{h,i,j}=c_{h,i,j}a_{s(h)_i}\quad(h\in H^1, 1\leq i\leq |s(h)|,  1\leq j\leq |r(h)|)$,
\medskip
\item $\sum\limits_{k=1}^{|r(h)|}b_{h,i,k}c_{h,j,k}= \delta_{ij}a_{s(h)_i}\quad(h\in H^1, ~1\leq i,j\leq |s(h)|)$ and
\medskip
\item $\sum\limits_{k=1}^{|s(h)|}c_{h,k,i}b_{h,k,j}= \delta_{ij}a_{r(h)_i}\quad(h\in H^1,~1\leq i,j\leq |r(h)|)$.
\end{enumerate}
We call $X$ an {\it $H$-family} in $A$. By the relations defining $L_K(H)$, there exists a unique $K$-algebra homomorphism $\phi: L_K(H)\rightarrow A$ such that $\phi(v)=a_v$, $\phi(h_{ij})=b_{h,i,j}$ and $\phi(h^*_{ij})=c_{h,i,j}$ for all $v\in H^0$, $h\in H^1$, $1\leq i\leq |s(h)|$ and $1\leq j\leq |r(h)|$. We will refer to this as the {\it Universal Property of $L_K(H)$}.
\end{enumerate}
\end{remark}

\begin{example}\label{ex1}
Consider the hypergraph $H=(H^0, H^1, s, r)$ where $H^0=\{v_1,v_2,w_1,w_2\}$, $H^1=\{h\}$, $s(h)=\{v_1,v_2\}$ and $r(h)=\{w_1,w_2\}$. We visualize $H$ as follows:\\
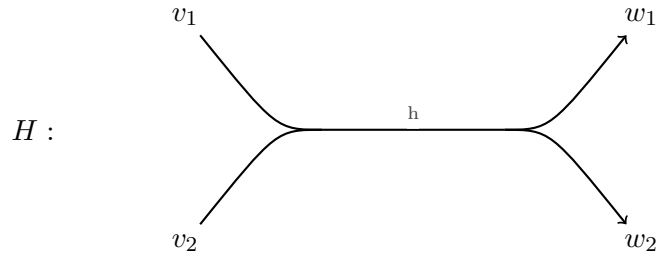
\begin{center}
\begin{tikzpicture}
  \node (H) at (0,-1.5) { $H:$ };
  \node (v_1) at (2,0)   { $v_1$ };
  \node (v_2) at (2,-3) { $v_2$ };
  \node (w_1) at (8,0)  { $w_1$ };
  \node (w_2) at (8,-3)  { $w_2$ };
  \connectFour[
  @ratio=.5,
  @pos1=.7,
  @pos2=.7,
  @edge 3=->,
  @edge 4=->,
  @edge=thick
  ]{v_1}{v_2}{w_1}{w_2}{h}
\end{tikzpicture}
\end{center}
$~$\\
The Leavitt path algebra $L_K(H)$ of $H$ is the $K$-algebra presented by the generating set 
\[\{v_i, w_i,h_{ij},h_{ij}^*\mid 1\leq i,j\leq 2\}\]
and the relations
\begin{enumerate}[(i)]
\item $uu'=\delta_{uu'}u\quad(u,u'\in \{v_1,v_2,w_1,w_2\})$,
\item $v_ih_{ij}=h_{ij}=h_{ij}w_j,~w_jh_{ij}^*=h_{ij}^*=h_{ij}^*v_i\quad(h\in H^1, 1\leq i,j\leq 2)$,
\item $\sum\limits_{k=1}^{2}h_{ik}h_{jk}^*= \delta_{ij}v_i\quad(1\leq i,j\leq 2)$ and
\item $\sum\limits_{k=1}^{2}h_{ki}^*h_{kj}= \delta_{ij}w_i\quad(1\leq i,j\leq 2)$.
\end{enumerate}
\end{example}


\begin{example}\label{ex3}
Let $(E,C)$ be a finitely separated graph (see \cite{aragoodearl}). Write each $X\in C$ as $X=\{e_{X,1},\dots,e_{X,n_X}\}$. Define a hypergraph $H=(H^0,H^1,s',r')$ by $H^0=E^0$, $H^1=\{h_X\mid X\in C\}$, $s'(h_X)=\{s(e_{X,1})\}$ and $r'(h_X)=\{r(e_{X,1}),\dots,r(e_{X,n_X})\}$. Then $L_K(E,C)\cong L_K(H)$.
\end{example}

By a {\it vertex-weighted graph} we mean a weighted graph in the sense of \cite{preusser2} such that edges emitted by the same vertex have the same weight. By \cite[Examples 5,6]{preusser2} the class of Leavitt path algebras of row-finite vertex-weighted graphs is big enough to contain all Leavitt path algebras of row-finite directed graphs and all Leavitt algebras $L_K(n,n+k)$.

\begin{example}\label{ex4}
Let $(E,w)$ be a row-finite vertex-weighted graph. Let $\Reg(E)$ denote the set of all vertices in $E$ that emit at least one edge. For each $v\in \Reg(E)$ write $s^{-1}(v)=\{e_{v,1},\dots,e_{v,n_v}\}$. Define a hypergraph $H=(H^0,H^1,s',r')$ by $H^0=E^0$, $H^1=\{h_v\mid v\in \Reg(E)\}$, $s'(h_v)=\{\underbrace{v,\dots,v}_{w(v)~times}\}$ and $r'(h_v)=\{r(e_{v,1}),\dots,r(e_{v,n_v})\}$. Then $L_K(E,w)\cong L_K(H)$.
\end{example}

Recall that a ring $R$ is said to have a {\it set of local units} $X$ in case $X$ is a set of idempotents in $R$ having the property that for each finite subset $S\subseteq R$ there exists an $x\in X$ such that $xsx=s$ for any $s\in S$.

\begin{proposition}\label{propprop}
Let $H$ be a hypergraph. Then:
\begin{enumerate}[(i)]
\item $L_K(H)$ has a set of local units, namely the set of all finite sums of distinct elements of $H^0$. If $H$ is finite, then $L_K(H)$ is a unital ring with $\sum\limits_{v\in H^0} v$ as multiplicative identity.
\item There is an involution $*$ on $L_K(H)$ mapping $k\mapsto k$, $v\mapsto v$, $h_{ij}\mapsto h_{ij}^*$ and $h_{ij}^*\mapsto h_{ij}$ for any $k\in K$, $v\in H^0$, $h\in H^1$, $1\leq i\leq |s(h)|$ and $1\leq j \leq |r(h)|$. 
\item Set $n:=\sup\{|s(h)|\mid h \in H^{1}\}$ (note that $n$ might be infinity). One can define a $\mathbb Z^n$-grading on $L_K(H)$ by setting $\deg(v):=0$, $\deg(h_{ij}):=\alpha_i$ and $\deg(h_{ij}^*):=-\alpha_i$ for any $v\in E^0$, $h \in E^{1}$, $1\leq i\leq |s(h)|$ and $1\leq j \leq |r(h)|$. Here $\alpha_i$ denotes the element of $\mathbb Z^n$ whose $i$-th component is $1$ and whose other components are $0$.
\end{enumerate}
\end{proposition}
\begin{proof}
\begin{enumerate}[(i)]
\item Follows from the fact that the elements of $H^0$ are mutually orthogonal idempotents in $L_K(H)$ such that $\sum\limits_{v\in H^0}vL_K(H)=\sum\limits_{v\in H^0}L_K(H)v=L_K(H)$.
\item Set $X:=\{v,h_{ij},h_{ij}^*\mid v\in H^0, h\in H^1, 1\leq i\leq |s(h)|,  1\leq j\leq |r(h)|\}$ and let $K\X$ denote the free $K$-algebra generated by $X$. Clearly there is a uniquely determined involution $*$ on $K\X$ mapping $k\mapsto k$, $v\mapsto v$, $h_{ij}\mapsto h_{ij}^*$ and $h_{ij}^*\mapsto h_{ij}$ for any $k\in K$, $v\in H^0$, $h\in H^1$, $1\leq i\leq |s(h)|$ and $1\leq j \leq |r(h)|$. Since the  set of the relations (i)-(iv) in Definition \ref{defhlpa} is clearly invariant under $*$, the involution $*$ induces an involution on $L_K(H)$.
\item Let $K\X$ be defined as above. Clearly there is a $\mathbb Z^n$-grading on $K\X$ defined by $\deg(v):=0$, $\deg(h_{ij}):=\alpha_i$ and $\deg(h_{ij}^*):=-\alpha_i$ for any $v\in E^0$, $h \in E^{1}$, $1\leq i\leq |s(h)|$ and $1\leq j \leq |r(h)|$. Since the relations (i)-(iv) in Definition \ref{defhlpa} are clearly homogeneous, the $\mathbb Z^n$-grading on $K\X$ induces a $\mathbb Z^n$-grading on $L_K(H)$. 
\end{enumerate}
\end{proof}

\begin{remark}
In the following we will refer to the grading defined in Proposition \ref{propprop} (iii) as the {\it standard grading} of $L_K(H)$. The isomorphisms $L_K(E,C)\cong L_K(H)$ and $L_K(E,w)\cong L_K(H)$ in Examples \ref{ex3} and \ref{ex4} are graded isomorphisms with respect to the standard gradings of $L_K(E,C)$, $L_K(E,w)$ and $L_K(H)$ (see \cite[Remark 2.13]{aragoodearl} and \cite[Proposition 5.7]{hazrat13}). 
\end{remark}

\section{Direct limits}
If $X$ and $Y$ are sets, $f:X\rightarrow Y$ is a map and $M$ is a multiset over $X$, then we define $f(M)$ as the multiset over $Y$ such that $f(M)(y)=\sum\limits_{x\in f^{-1}(y)}M(x)$. Note that if $M\in \Mult(X)$, then $|M|=|f(M)|$ and $|M_{\supp}|\leq |f(M)_{\supp}|$.
\begin{definition}
Let $H$ and $I$ be hypergraphs. A {\it hypergraph homomorphism} $\phi:H\rightarrow I$ consists of two maps $\phi^0:H^0\rightarrow I^0$ and $\phi^1:H^1\rightarrow I^1$ such that $s(\phi^1(h))=\phi^0(s(h))$ and $r(\phi^1(h))=\phi^0(r(h))$ for any $h\in H^1$. We denote by $\Hy$ the category whose objects are all hypergraphs and whose morphisms are all hypergraph homomorphisms. 
\end{definition}

\begin{proposition}
The category $\Hy$ admits arbitrary direct limits.
\end{proposition}
\begin{proof}
Let $\{H_i,\phi_{ij}\mid i,j\in I, i\leq j\}$ be a direct system in $\Hy$. For any $l\in \{0,1\}$ we define $H^l$ as the direct limit of the $H^l_i$'s in the category of sets. We identify $H^l$ with the set $\bigsqcup\limits_{i\in I}H_i^l/\hspace{-0.12cm}\sim_l$ where $(x,i)\sim_l(y,j)\Leftrightarrow~\exists k\in I: \phi_{ik}^l(x)=\phi_{jk}^l(y)$. For a $[(h,i)]\in H^1$ we set 
\[s_H([(h,i)]):=\{[(s_{H_i}(h)_1,i)],\dots,[(s_{H_i}(h)_{|s_{H_i}(h)|},i)]\}\]
and 
\[r_H([(h,i)]):=\{[(r_{H_i}(h)_1,i)],\dots,[(r_{H_i}(h)_{|r_{H_i}(h)|},i)]\}.\]
One checks easily that one gets well-defined maps $s_H,r_H:H^1\rightarrow \Mult(H^0)$. Set $H:=(H^0, H^1, s_H, r_H)$. For any $i\in I$ let $\phi^0_i:H^0_i\rightarrow H^0$ and $\phi^1_i:H^1_i\rightarrow H^1$ be the canonical maps and set $\phi_i:=(\phi_i^0, \phi_i^1)$. It is routine to check that $\phi_i:H_i\rightarrow H$ is a hypergraph homomorphism for any $i\in I$ and that $\{H,\phi_i\mid i\in I\}$ is a direct limit of the direct system $\{H_i,\phi_{ij}\mid i,j\in I, i\leq j\}$.
\end{proof}

\begin{definition}
Let $H$ and $I$ be a hypergraphs. Then $I$ is called a {\it subhypergraph} of $H$ if $I^0\subseteq H^0$, $I^1\subseteq H^1$ and $s_I$ and $r_I$ are the restrictions of $s_H$ resp. $r_H$ to $I^0$.
\end{definition}

\begin{proposition}\label{proplim}
Let $H$ be a hypergraph. Then $H$ is a direct limit of the direct system of all finite subhypergraphs of $H$.
\end{proposition}
\begin{proof}
Straightforward.
\end{proof}

\begin{definition}
In Definition \ref{defhlpa} we associated to any hypergraph $H$ a $K$-algebra $L_K(H)$. If $\phi:H\rightarrow I$ is a morphism in $\Hy$, then there is a unique $K$-algebra homomorphism $L_K(\phi):L_K(H)\rightarrow L_K(I)$ such that $L_K(\phi)(v)=\phi^0(v)$, $L_K(\phi)(h_{ij})=(\phi^1(h))_{ij}$ and $L_K(\phi)(h_{ij}^*)=(\phi^1(h))_{ij}^*$ for any $v\in H^0$, $h\in H^1$, $1\leq i\leq |s(h)|$ and $1\leq j \leq |r(h)|$ (follows from the Universal Property of $L_K(H)$, see Remark \ref{remhlpa} (c)). Let $\A$ denote the category of $K$-algebras with local units. One checks easily that $L_K:\Hy \rightarrow \A$ is a functor that commutes with direct limits.
\end{definition}

\begin{proposition}\label{proplim2}
Let $H$ be a hypergraph. Then $L_K(H)$ is a direct limit of a direct system of Leavitt path algebras of finite hypergraphs.
\end{proposition}
\begin{proof}
Follows from Proposition \ref{proplim} and the fact that $L_K$ commutes with direct limits.
\end{proof}

\section{Linear bases}
Throughout this section $H$ denotes a hypergraph. Our goal is to find a basis for the $K$-vector space $L_K(H)$. We denote the directed graph associated to $H$ by $E$ and the double graph of $E$ by $\hat E$ (see Remark \ref{remhlpa} (b)). We set $X:=\hat E^0\cup \hat E^1$, $\X:=\{\text{nonempty words over }X\}$ and $\overline{\X}:=\X\cup\{\text{empty word}\}$. Together with juxtaposition $\X$ is a semigroup and $\overline{\X}$ a monoid. If $A,B\in \overline{\X}$, then we call $B$ a {\it subword} of $A$ if there are $C,D\in\overline{\X}$ such that $A=CBD$. We denote by $K\X$ the free $K$-algebra generated by $X$ (i.e. $K\X$ is the $K$-vector space with basis $\X$ which becomes a $K$-algebra by linear extending the juxtaposition of words). Note that $L_K(H)$ is the quotient of $K\X$ by the relations (i)-(iv) in Definition \ref{defhlpa}. 

\begin{definition}\label{defpath}
Let $G=(G^0,G^1,s, r)$ be a directed graph. A {\it path} in $G$ is a nonempty word $p=x_1\dots x_n$ over the alphabet $G^0\cup G^1$ such that either $x_i\in G^1~(i=1,\dots,n)$ and $r(x_i)=s(x_{i+1})~(i=1,\dots,n-1)$ or $n=1$ and $x_1\in G^0$. By definition, the {\it length} $|p|$ of $p$ is $n$ in the first case and $0$ in the latter case. We set $s(p):=s(x_1)$ and $r(p):=r(x_n)$ (here we use the convention $s(v)=v=r(v)$ for any $v\in G^0$). 
\end{definition}

We call a path in $\hat E$ a {\it d-path}. While the d-paths form a basis for the path algebra $K \hat E$, a basis for the Leavitt path algebra $L_K(H)$ is formed by the nod-paths, which we will define next. 

\begin{definition}\label{defnod}
The words
\[h_{i1}h_{j1}^*~(h\in H^1,1\leq i,j\leq |s(h)|)\text{ and }h_{1i}^*h_{1j}~(h\in H^1,1\leq i,j\leq |r(h)|)\]
in $\X$ are called {\it forbidden}. A {\it normal d-path} or {\it nod-path} is a d-path such that none of its subwords is forbidden. An element of $K\X$ is called {\it normal} if it lies in the linear span $K\X_{\nod}$ of all nod-paths.
\end{definition}

\begin{theorem}\label{thmbasis}
Any element of $L_K(H)$ has precisely one normal representative. Moreover, the map $\NF:L_K(H)\rightarrow K\X_{\nod}$ which associates to each element of $L_K(H)$ its normal representative is an isomorphism of $K$-vector spaces.
\end{theorem}
\begin{proof}
In order to be able to apply \cite[Theorem 15]{hazrat-preusser}, we replace the relations (i)-(iv) in Definition~\ref{defhlpa} by the relations (i')-(v') below.
\begin{enumerate}[(i')]
\item For any $v,w \in H^0$, \[vw = \delta_{vw}v.\]
\medskip
\item For any $v \in H^0$, $h\in H^1$, $1\leq i \leq |s(h)|$ and $1\leq j \leq |r(h)|$,
\begin{align*}
vh_{ij}&=\delta_{vs(h)_i}h_{ij},\\
h_{ij}v&=\delta_{vr(h)_j}h_{ij},\\
vh_{ij}^*&=\delta_{vr(h)_j}h_{ij}^* \text{ and}\\
h_{ij}^*v&=\delta_{vs(h)_i}h_{ij}^*.
\end{align*}
\medskip 
\item For any $g,h\in H^1$, $1\leq i \leq |s(g)|$, $1\leq j \leq |r(g)|$, $1\leq k \leq |s(h)|$ and $1\leq l \leq |r(h)|$,
\begin{align*}
g_{ij}h_{kl}&=0\text{ if  } r(g)_j\neq s(h)_k,\\
g_{ij}^*h_{kl}&=0\text{ if  } s(g)_i\neq s(h)_k,\\
g_{ij}h_{kl}^*&=0\text{ if  } r(g)_j\neq r(h)_l\text{ and}\\
g_{ij}^*h_{kl}^*&=0\text{ if  } s(g)_i\neq r(h)_l.
\end{align*}
\medskip 
\item For all $h\in H^1$ and $1\leq i,j\leq |s(h)|$, 
\[h_{i1}h_{j1}^*= \delta_{ij}s(h)_i-\sum\limits_{k=2}^{|r(h)|}h_{ik}h_{jk}^*.\]
\medskip 
\item For all $h\in H^1$ and $1\leq i,j\leq |r(h)|$, 
\[h_{1i}^*h_{1j}= \delta_{ij}r(h)_i-\sum\limits_{k=2}^{|s(h)|}h_{ki}^*h_{kj}.\]
\end{enumerate}
Clearly the relations (i')-(v') above generate the same ideal $I$ of $K\X$ as the relations (i)-(iv) in Definition~\ref{defhlpa}. Denote by $S$ the reduction system for $K\X$ defined by the relations (1')-(5') (i.e. $S$ is the set of all pairs $\sigma=(W_\sigma,f_\sigma)$ where $W_\sigma$ equals the left hand side of an equation in (i')-(v') and $f_\sigma$ the corresponding right hand side).\\
For any $A=x_1\dots x_n\in \X$ set $l(A):=n$ and $m(A):= \big |\{i\in\{1,\dots,n-1\}|x_ix_{i+1} \text{ is forbidden}\} \big |$. Define a partial ordering $\leq$ on $\X$ by 
\[A\leq B\Leftrightarrow \big [A=B\big ]~\lor~\big[l(A)<l(B)\big]~\lor \big[l(A)=l(B)~\land~ \forall C,D\in \overline\X:m(CAD)<m(CBD)\big].\]
Clearly $\leq$ is a semigroup partial ordering on $\X$ compatible with $S$ and the descending chain condition is satisfied.\\
It remains to show that all ambiguities of $S$ are resolvable. In the table below we list all types of ambiguities which may occur.
\begin{table}[!htbp]
\resizebox{17.7cm}{!}{
\begin{tabular}{|c||c|c|c|c|c|}
\hline &(1')&(2')&(3')&(4')&(5')\\
\hhline{|=#=|=|=|=|=|} (1')&$uvw$&$vwh_{ij}$,$vwh_{ij}^*$ &-&-&-\\
\hline (2')&$h_{ij}vw$, $h_{ij}^*vw$&$vh_{ij}w$, $vh_{ij}^*w$, $g_{ij}vh_{kl}$ etc.&$vg_{ij}h_{kl}$, $vg_{ij}h_{kl}^*$ etc.&$vh_{i1}h_{j1}^*$&$vh_{1i}^*h_{1j}$\\
\hline (3')&-&$g_{ij}h_{kl}v$, $g_{ij}^*h_{kl}v$ etc.&$f_{ij}g_{kl}h_{pq}$, $f_{ij}g_{kl}h_{pq}^*$ etc.&$g_{kl}h_{i1}h_{j1}^*$, $g_{kl}^*h_{i1}h_{j1}^*$&$g_{kl}h_{1i}^*h_{1j}$, $g_{kl}^*h_{1i}^*h_{1j}$\\
\hline (4')&-&$h_{i1}h_{j1}^*v$&$h_{i1}h_{j1}^*g_{kl}$, $h_{i1}h_{j1}^*g_{kl}^*$&-&$h_{i1}h_{11}^*h_{1j}$\\
\hline (5')&-&$h_{1i}^*h_{1j}v$&$h_{1i}^*h_{1j}g_{kl}$, $h_{1i}^*h_{1j}g_{kl}^*$&$h_{1i}^*h_{11}h_{j1}^*$&-\\
\hline
\end{tabular}
}
\end{table}\\
Note that there are no inclusion ambiguities. The (4')-(5') and (5')-(4') ambiguities $h_{i1}h_{11}^*h_{1j}$ and $h_{1i}^*h_{11}h_{j1}^*$ are the ones which are most difficult to resolve. We show how to resolve the ambiguity $h_{i1}h_{11}^*h_{1j}$ (where $h\in H^1$, $1\leq i \leq |s(h)|$ and $1\leq j \leq |r(h)|$) and leave the other ambiguities to the reader.
\xymatrixcolsep{-7pc}
\xymatrixrowsep{4pc}
\[\xymatrix{
&h_{i1}h_{11}^*h_{1j}\ar[rd]^-{(5')}\ar[ld]_-{(4')}&
\\
{\begin{array}{cc}&(\delta_{i1}s(h)_i-\sum\limits_{k=2}^{|r(h)|}h_{ik}h_{1k}^*)h_{1j}\\=&\delta_{i1}s(h)_ih_{1j}-\sum\limits_{k=2}^{|r(h)|}h_{ik}h_{1k}^*h_{1j}\end{array}}\ar[d]_-{(5')}
&&
{\begin{array}{cc}&h_{i1}(\delta_{1j}r(h)_1-\sum\limits_{k=2}^{|s(h)|}h_{k1}^*h_{kj})\\=&\delta_{1j}h_{i1}r(h)_1-\sum\limits_{k=2}^{|s(h)|}h_{i1}h_{k1}^*h_{kj}\end{array}\ar[d]^-{(4')}}
\\
{\begin{array}{cc}&\delta_{i1}s(h)_ih_{1j}\\&-\sum\limits_{k=2}^{|r(h)|}h_{ik}(\delta_{kj}r(h)_k-\sum\limits_{l=2}^{|s(h)|}h^*_{lk}h_{lj})\\=&\delta_{i1}s(h)_ih_{1j}-\sum\limits_{k=2}^{|r(h)|}\delta_{kj}h_{ik}r(h)_k\\&+\sum\limits_{k=2}^{|r(h)|}\sum\limits_{l=2}^{|s(h)|}h_{ik}h^*_{lk}h_{lj}\end{array}}\ar[rd]_-{(2')}
&&
{\begin{array}{cc}&\delta_{1j}h_{i1}r(h)_1\\&-\sum\limits_{k=2}^{|s(h)|}(\delta_{ik}s(h)_i-\sum\limits_{l=2}^{|r(h)|}h_{il}h_{kl}^*)h_{kj}\\=&\delta_{1j}h_{i1}r(h)_1-\sum\limits_{k=2}^{|s(h)|}\delta_{ik}s(h)_ih_{kj}\\&+\sum\limits_{k=2}^{|s(h)|}\sum\limits_{l=2}^{|r(h)|}h_{il}h_{kl}^*h_{kj}\end{array}}\ar[ld]^-{(2')}
\\
&{\begin{array}{cc}&\delta_{i1}h_{1j}-\sum\limits_{k=2}^{|r(h)|}\delta_{kj}h_{ik}+\sum\limits_{k=2}^{|r(h)|}\sum\limits_{l=2}^{|s(h)|}h_{ik}h^*_{lk}h_{lj}\\=&\delta_{1j}h_{i1}-\sum\limits_{k=2}^{|s(h)|}\delta_{ik}h_{kj}+\sum\limits_{k=2}^{|s(h)|}\sum\limits_{l=2}^{|r(h)|}h_{il}h_{kl}^*h_{kj}\end{array}}&}
\]
(note that $\sum\limits_{k=1}^{|r(h)|}\delta_{kj}h_{ik}=h_{ij}=\sum\limits_{k=1}^{|s(h)|}\delta_{ik}h_{kj}$).
It follows from \cite[Theorem 15]{hazrat-preusser}, that $K\X_{\irr}$ is a set of representatives for the elements of $K\X/I=L_K(H)$. Clearly $K\X_{\irr}=K\X_{\nod}$.\\
Clearly the map $\NF:L_K(H)\rightarrow K\X_{\nod}$ which associates to each element of $L_K(H)$ its normal representative is bijective. That $\NF$ is linear follows from \cite[Lemma 1.1]{bergman78} (note that $\NF=r_S$). 
\end{proof}

\begin{corollary}\label{corbasis}
The images of the nod-paths in $L_K(H)$ form a basis of the $K$-vector space $L_K(H)$.
\end{corollary}
\begin{proof}
By Theorem \ref{thmbasis}, the map $\NF:L_K(H)\rightarrow K\X_{\nod}$ is an isomorphism of $K$-vector spaces. Its inverse is the map $K\X_{\nod}\rightarrow L_K(H)$ induced by the inclusion map $K\X_{\nod}\hookrightarrow K\X$. Since the nod-paths form a $K$-basis for $K\X_{\nod}$, the assertion of the corollary follows.
\end{proof}

\section{The Gelfand-Kirillov dimension}
First we want to recall some general facts on the growth of algebras. Let $A\neq\{0\}$ be a finitely generated $K$-algebra. Let $V$ be a {\it finite-dimensional generating subspace} of $A$, i.e. a finite-dimensional subspace of $A$ that generates $A$ as a $K$-algebra. For $n\geq 1$ let $V^n$ denote the linear span of the set $\{v_1\dots v_k\mid k\leq n, v_1,\dots,v_k\in V\}$. Then 
\[V =V^1\subseteq V^2\subseteq V^3\subseteq \dots, \quad A =\bigcup\limits_{n\in \N}V^n\text{ and }d_V(n):=\dim V^n<\infty.\] 
Given functions $f, g:\N\rightarrow \R^+$, we write $f\preccurlyeq g$ if there is a $c\in\N$ such that $f(n)\leq cg(cn)$ for all $n$. If $f\preccurlyeq g$ and $g\preccurlyeq f$, then the functions $f, g$ are called {\it asymptotically equivalent} and we write $f\sim g$. If $W$ is another finite-dimensional generating subspace of $A$, then $d_V\sim d_W$. The {\it Gelfand-Kirillov dimension} or {\it GK dimension} of $A$ is defined as
\[\GKdim A := \limsup\limits_{n\rightarrow \infty}\log_nd_V(n).\]
The definition of the GK dimension does not depend on the choice of the finite-dimensional generating subspace $V$. If $d_V\preccurlyeq n^m$ for some $m\in \N$, then $A$ is said to have {\it polynomial growth} and we have $\GKdim A \leq m$. If $d_V\sim a^n$ for some real number $a>1$, then $A$ is said to have {\it exponential growth} and we have $\GKdim A =\infty$. If $A$ does not happen to be finitely generated over $K$, then the GK dimension of $A$ is defined as
\[\GKdim(A) := \sup\{\GKdim(B)\mid B \text{ is a finitely generated subalgebra of }A\}.\]

\begin{definition}\label{deffnodc}
Let $H$ be a hypergraph. Let $p$ and $q$ be nod-paths. If there is a nod-path $o$ such that $p$ is not a prefix of $o$ and $poq$ is a nod-path, then we write $p\overset{\nod}{\Longrightarrow} q$. If $pq$ is a nod-path or $p\overset{\nod}{\Longrightarrow} q$, then we write $p\Longrightarrow q$.
\end{definition}


\begin{definition}\label{deffnod2}
Let $H$ be a hypergraph. A {\it nod$^2$-path} is a nod-path $p$ such that $p^2$ is a nod-path. A {\it quasi-cycle} is a nod$^2$-path $p$ such that none of the subwords of $p^2$ of length $<|p|$ is a nod$^2$-path. A quasi-cycle $p$ is called {\it selfconnected} if $p\overset{\nod}{\Longrightarrow}p$.
\end{definition}
 
\begin{remark}\label{remnod2}
$~$\vspace{-0.2cm}
\begin{enumerate}[(a)]
\item
Let $p=x_1\dots x_n$ be a quasi-cycle. Assume that $x_i=x_j$ for some $1\leq i <j\leq n$. Then we get the contradiction that $x_i\dots x_{j-1}$ is a nod$^2$-path of length $<n$. Hence $x_i\neq x_j$ for all $i\neq j$. It follows that there is only a finite number of quasi-cycles if $(E,w)$ is finite.
\item
If $x_1\dots x_n$ is a nonempty word over some alphabet, then we call the words $x_{m+1}\dots x_{n}x_1\dots x_{m}~(1\leq m\leq n)$ {\it shifts} of $x_1\dots x_n$. One checks easily that any shift of a quasi-cycle $p$ is a quasi-cycle (note that if $q$ is a shift of $p$, then any subword of $q^2$ of length $<|q|=|p|$ is also a subword of $p^2$). If $p$ and $q$ are quasi-cycles, then we write $p\approx q$ iff $q$ is a shift of $p$. Clearly $\approx$ is an equivalence relation on the set of all quasi-cycles.
\item 
Let $p=x_1\dots x_n$ be a quasi-cycle. Then $p^*:=x_n^*\dots x_1^*$ is a quasi-cycle.
\end{enumerate}
\end{remark}

The following lemma shows, that quasi-cycles behave like cycles in a way (one cannot ``take a shortcut").
\begin{lemma}\label{lemshc}
Let $H$ be a hypergraph, $p=x_1\dots x_n$ a quasi-cycle and $1\leq i,j\leq n$. Then $x_ix_j$ is a nod-path iff $i<n$ and $j=i+1$ or $i=n$ and $j=1$.
\end{lemma}
\begin{proof}
If $i<n$ and $j=i+1$ or $i=n$ and $j=1$, then clearly $x_ix_j$ is a nod-path. Suppose now that $x_ix_j$ is a nod-path. \\
\\
\underline{case 1} Suppose $i=j$. Assume that $n>1$. Then we get the contradiction that $x_i$ is a nod$^2$-path which is a subword of $p^2$ of length $1<|p|=n$. Hence $n=1$ and we have $i=j=1=n$.\\
\\
\underline{case 2} Suppose $i<j$. Then $x_j\dots x_nx_1\dots x_i$ is a nod$^2$-path which is a subword of $p^2$ of length $n-j+1+i$. It follows that $j=i+1$.\\
\\
\underline{case 3} Suppose $j<i$. Then $x_j\dots x_i$ is a nod$^2$-path which is a subword of $p^2$ of length $i-j+1$. It follows that $j=1$ and $i=n$.
\end{proof}

\begin{lemma}\label{lemexp}
Let $H$ be a finite hypergraph. If there is a selfconnected quasi-cycle $p$, then $L_K(H)$ has exponential growth.
\end{lemma}
\begin{proof}
Let $o$ be a nod-path such that $p$ is not a prefix of $o$ and $pop$ is a nod-path. Let $n\in \N$. Consider the nod-paths
\begin{equation}
p^{i_1}op^{i_2}\dots op^{i_k}
\end{equation}
where $k,i_1,\dots,i_k\in \N$ satisfy
\begin{equation}
(i_1+\dots+i_k)|p|+(k-1)|o|\leq n.
\end{equation}
Let $A=(k,i_1,\dots,i_k)$ and $B=(k',i'_1,\dots,i'_{k'})$ be different solutions of (2). Assume that $A$ and $B$ define the same nod-path in (1). After cutting out the common beginning, we can assume that the nod-path defined by $A$ starts with $o$ and the nod-path defined by $B$ with $p$ or vice versa. Since $p$  is not a prefix of $o$, it follows that $|p|>|o|$. Write $p=x_1\dots x_m$. Since the next letter after an $o$ must be a $p$, we get $x_{|o|+1}=x_1$ which contradicts Remark \ref{remnod2}(a). Hence different solutions of (2) define different nod-paths in (1). Let $V$ denote the finite-dimensional generating subspace of $L_K(H)$ spanned by the set $\{v,h_{ij},h_{ij}^*\mid v\in H^0, h\in H^1, 1\leq i\leq |s(h)|, 1\leq j\leq |r(h)|\}$. By Theorem \ref{thmbasis} the nod-paths in (1) are linearly independent in $V^n$. The number of solutions of (2) is $\sim 2^n$ and hence $L_K(H)$ has exponential growth. 
\end{proof}

In Definition \ref{defcondab} below, we introduce the Condition (A') for hypergraphs. This condition will be used again in Section 9.

We call a multiset $M:X\rightarrow \N_0$ a {\it set} and identify it with $M_{\supp}$ if $M(x)\leq 1$ for any $x\in X$. If the multiset $M$ is not a set, then we call it a {\it proper multiset}. 

\begin{definition}\label{defcondab}
$H$ satisfies {\it Condition (A')} if there is an $h\in H^1$ such that $|s(h)|, |r(h)|\geq 2$ and moreover $s(h)$ is a proper multiset or $r(h)$ is a proper multiset or $s(h)$ and $r(h)$ are sets with nonempty intersection.
\end{definition}

\begin{corollary}\label{corexp}
If $H$ is a finite hypergraph that satisfies Condition (A'), then $L_K(H)$ has exponential growth. 
\end{corollary}
\begin{proof}
Since $H$ satisfies Condition (A'), we can choose an $h\in H^1$ such that $|s(h)|, |r(h)|\geq 2$ and $s(h)$ is a proper multiset or $r(h)$ is a proper multiset or $s(h)$ and $r(h)$ are sets with nonempty intersection.\\
First suppose that $s(h)_{\supp}\cap r(h)_{\supp}\neq\emptyset$. By Remark \ref{remhlpa} (a) we may assume that $s(h)_2=r(h)_2$. Clearly $p:=h_{22}$ is a quasi-cycle. Further $h_{22}h_{22}^*h_{22}$ is a nod-path and therefore $p$ is selfconnected. Hence, by the previous lemma, $L_K(H)$ has exponential growth.\\
Now suppose that $s(h)_{\supp}\cap r(h)_{\supp}=\emptyset$ and $s(h)$ is a proper multiset. Choose $1\leq i<j\leq |s(h)|$ such that $s(h)_i=s(h)_j$. Clearly $p:=h_{j2}h_{j2}^*$ is a quasi-cycle. Further $h_{j2}h_{j2}^*h_{j2}h_{i2}^*h_{j2}h_{j2}^*$ is a nod-path and therefore $p$ is selfconnected. Hence, by the previous lemma, $L_K(H)$ has exponential growth. The case that $s(h)_{\supp}\cap r(h)_{\supp}=\emptyset$ and $r(h)$ is a proper multiset can be handled similarly.
\end{proof}

If $H$ is a hypergraph, then we denote by $E'$ the subset of $\{h_{ij},h_{ij}^*\mid h\in H^1, 1\leq i\leq |s(h)|, 1\leq j\leq |r(h)|\}$ consisting of all the elements which are not a letter of a quasi-cycle. We denote by $P'$ the set of all nod-paths which are composed from elements of $E'$.

\begin{lemma}\label{lemfnt}
Let $H$ be a finite hypergraph. Then $|P'|<\infty$.
\end{lemma}
\begin{proof}
Let $p'=x_1\dots x_n\in P'$. Assume that there are $1\leq i <j\leq n$ such that $x_i=x_j$. Then $x_i\dots x_{j-1}$ is a nod$^2$-path. Since for any nod$^2$-path $q$ which is not a quasi-cycle there is a shorter nod$^2$-path $q'$ such that any letter of $q'$ already appears in $q$, we get a contradiction. Hence the $x_i$'s are pairwise distinct. It follows that $|P'|<\infty$ since $|E'|<\infty$.
\end{proof}

Let $H$ be a hypergraph. A sequence $p_1,\dots,p_k$ of quasi-cycles such that $p_i\not\approx p_j$ for any $i\neq j$ is called a {\it chain of length $k$} if $p_1 \Longrightarrow p_2\Longrightarrow \dots \Longrightarrow p_k$. We call a nod-path $p$ {\it trivial} if $p=v$ for some $v\in H^0$ and {\it nontrivial} otherwise.

\begin{lemma}\label{lemnew}
Let $H$ be a hypergraph. If there is no selfconnected quasi-cycle, then any nontrivial nod-path $\alpha$ can be written as
\begin{equation}
\alpha=o_1p_1^{l_1}q_1o_2p_2^{l_2}q_2\dots o_kp_k^{l_k}q_ko_{k+1}
\end{equation}
where $k\geq 0$, $o_i$ is the empty word or $o_i\in P'~(1\leq i \leq k+1)$, $p_1,\dots,p_k$ is a chain of quasi-cycles, $l_i$ is a nonnegative integer $(1\leq i \leq k)$, and $q_i\neq p_i$ is a prefix of $p_i~(1\leq i \leq k)$.
\end{lemma}
\begin{proof}
Let $\alpha=x_1\dots x_r$ be a nontrivial nod-path. Then $x_1,\dots,x_r \in \{e_i,e_i^*\mid e\in E^1, 1\leq i\leq w(e)\}$. Let $r_1$ be minimal such that $x_{r_1}\not\in E'$ (if such an $r_1$ does not exist, then $\alpha\in P'$ and (3) holds with $k=0$ and $o_1=\alpha$). Then $o_1:=x_1\dots x_{r_1-1}$ is either the empty word or $o_1\in P'$. Since $x_{r_1}\not\in E'$, we have that $x_{r_1}$ is a letter of a quasi-cycle, say $p_1$. By Remark \ref{remnod2}(b) we may assume that $x_{r_1}$ is the first letter of $p_1$. Let $r_2>r_1$ be minimal such that $x_{r_2}$ is not a letter of $p_1$. Then $x_{r_1},\dots, x_{r_2-1}$ are letters of $p_1$. It follows from Lemma \ref{lemshc} that $x_{r_1}\dots x_{r_2-1}=p_1^{l_1}q_1$ for some $l_1\geq 0$ and $q_1\neq p_1$ that is a prefix of $p_1$. Let $r_3\geq r_2$ be minimal such that $x_{r_3}\not\in E'$ and set $o_2:=x_{r_2}\dots x_{r_3-1}$. Then $o_2$ is either the empty word or $o_2\in P'$. Since $x_{r_3}\not \in E'$, we have that $x_{r_3}$ is a letter of a quasi-cycle, say $p_2$, and we may assume that $p_2$ starts with $x_{r_3}$. Let $r_4>r_3$ be minimal such that $x_{r_4}$ is not a letter of $p_2$. Then $x_{r_3}\dots x_{r_4-1}=p_2^{l_2}q_2$ for some $l_2\geq 0$ and $q_2\neq p_2$ that is a prefix of $p_2$ (see above).\\
By repeating the procedure described in the previous paragraph one gets that $\alpha$ can be written as in Equation (3) where $k\geq 0$, $o_i$ is the empty word or $o_i\in P'~(1\leq i \leq k+1)$, $p_1,\dots,p_k$ are quasi-cycles, $l_i$ is a nonnegative integer $(1\leq i \leq k)$, and $q_i\neq p_i$ is a prefix of $p_i~(1\leq i \leq k)$. It remains to show that $p_1,\dots,p_k$ is a chain of quasi-cycles.\\
First we show that $p_i \Longrightarrow p_{i+1}$ for any $1\leq i\leq k-1$. Because of (3) we know that  $p_i\underbrace{q_io_{i+1}}_{\beta:=}p_{i+1}$ is a nod-path. If $\beta$ is the empty word, then $p_i\Longrightarrow p_{i+1}$ by the definition of $\Longrightarrow$. Otherwise $p_i$ is not a prefix of $\beta$ since $q_i$ is the empty word or $|q_i|<|p_i|$ and $o_i$ is the empty word or $o_{i+1}\in P'$. Hence we get again $p_i\Longrightarrow p_{i+1}$.\\
Now assume that $p_i\approx p_j$ for some $1\leq i<j\leq n$. Write $p_i=y_1\dots y_t$. Then $p_j=y_{m+1}\dots y_{t}y_1\dots y_{m}$ for some $1\leq m\leq t$. Because of (3) we have that $p_iq_io_{i+1}p_{i+1}q_{i+1}\dots o_{j-1}p_{j-1}q_{j-1}o_jp_j$ is a nod-path. It follows that 
\[p_i\underbrace{q_io_{i+1}p_{i+1}q_{i+1}\dots o_{j-1}p_{j-1}q_{j-1}o_{j}y_{m+1}\dots y_{t}}_{\gamma:=}p_i\]
is a nod-path. By the construction in the first paragraph of this proof, the first letter in $\alpha$ after $q_i$ is not a letter of $p_i$. Hence $p_i$ is not a prefix of $\gamma$ and hence we get the contradiction $p_i\overset{\nod}{\Longrightarrow} p_i$. Therefore $p_i\not\approx p_j$ for any $i\neq j$ and thus $p_1,\dots,p_k$ is a chain of quasi-cycles.
\end{proof}
\begin{theorem}\label{thmgk}
Let $H$ be a finite hypergraph. Then:
\begin{enumerate}[(i)]
\item $L_K(H)$ has polynomial growth iff there is no selfconnected quasi-cycle.
\item If $L_K(H)$ has polynomial growth, then $\GKdim L_K(H)=d$ where $d$ is the maximal length of a chain of quasi-cycles.
\end{enumerate}
\end{theorem}
\begin{proof}
If there is a selfconnected quasi-cycle, then $L_K(H)$ has exponential growth by Lemma \ref{lemexp}. Suppose now that there is no selfconnected quasi-cycle. Let $V$ denote the finite-dimensional generating subspace of $L_K(H)$ spanned by $\{v,h_{ij},h_{ij}^*\mid v\in H^0, h\in H^1, 1\leq i\leq |s(h)|, 1\leq j\leq |r(h)|\}$. By Theorem \ref{thmbasis} the nod-paths of length $\leq n$ form a basis for $V^n$. 
By Lemma \ref{lemnew} we can write any nontrivial nod-path $\alpha$ of length $\leq n$ as
\begin{equation}
\alpha=o_1p_1^{l_1}q_1o_2p_2^{l_2}q_2\dots o_kp_k^{l_k}q_ko_{k+1}
\end{equation}
where $k\geq 0$, $o_i$ is the empty word or $o_i\in P'~(1\leq i \leq k+1)$, $p_1,\dots,p_k$ is a chain of quasi-cycles, $l_i$ is a nonnegative integer $(1\leq i \leq k)$, and $q_i\neq p_i$ is a prefix of $p_i~(1\leq i \leq k)$. Clearly 
\begin{equation}
l_1|p_1| +\dots+l_k|p_k| \leq n
\end{equation}
since $|\alpha|\leq n$. Now fix a chain $p_1,\dots,p_k$ of quasi-cycles and further $o_i$'s and $q_i$'s as above. The number of solutions $(l_1,\dots,l_k)$ of (5) is $\sim n^k$. This implies that the number of nod-paths $\alpha$ of length $n$ or less that can be written as in (4) (corresponding to the choice of the $p_i$'s, $o_i$'s and $q_i$'s) is $\preccurlyeq n^k \leq n^d$. Since there are only finitely many quasi-cycles and finitely many choices for the $o_i$'s and $q_i$'s (note that $|P'|<\infty$ by Lemma \ref{lemfnt}), the number of nod-paths of length $n$ or less is $\preccurlyeq n^d$.\\
On the other hand, choose a chain $p_1,\dots,p_d$ of length $d$. Then $p_1o_1p_2\dots o_{d-1}p_d$ is a nod-path for some $o_1,\dots,o_{d-1}$ such that for any $i\in\{1,\dots,d-1\}$, $o_i$ is either the empty word or a nod-path such that $p_i$ is not a prefix of $o_i$. Consider the nod-paths
\begin{equation}
p_1^{l_1}o_1p_2^{l_2}\dots o_{d-1}p_d^{l_d}
\end{equation}
where $l_1,\dots,l_d\in \N$ satisfy
\begin{equation}
l_1|p_1| +\dots+l_d|p_d|+|o_1|+\dots+|o_{d-1}|\leq n.
\end{equation}
Let $A=(l_1,\dots,l_d)$ and $B=(l'_1,\dots,l'_{d})$ be different solutions of (7). Assume that $A$ and $B$ define the same nod-path in (6). After cutting out the common beginning, we can assume that the nod-path defined by $A$ starts with $o_ip_{i+1}$ for some $i\in\{1,\dots,d-1\}$ and the nod-path defined by $B$ with $p_io_i$ or $p_i^2$. If $o_i$ is the empty word, then we get the contradiction $p_i=p_{i+1}$, since $p_i$ and $p_{i+1}$ are quasi-cycles. Suppose now that $o_i$ is not the empty word. Since $p_i$ is not a prefix of $o_i$, it follows that $|o_i|<|p_i|$. Further $|p_i|<|o_i|+|p_{i+1}|$ (otherwise $p_{i+1}$ would be a subword of $p_i$ of length $<|p_i|$). Write $p_i=x_1\dots x_k$ and $p_{i+1}=y_1\dots y_m$.\\
\\
\underline{case 1} Assume that $|p_i|\leq |p_{i+1}|$. Then $o_i=x_1\dots x_j$ and $p_{i+1}=x_{j+1}\dots x_k x_1\dots x_jy_{k+1}\dots y_m$ for some $j\in\{1,\dots,k-1\}$. By Remark \ref{remnod2}(b), $x_{j+1}\dots x_k x_1\dots x_j$ is a quasi-cycle. It follows that $k=m$. Hence we get the contradiction $p_i\approx p_{i+1}$.\\
\\
\underline{case 2} Assume that $|p_i|>|p_{i+1}|$. Then $o_i=x_1\dots x_j$ and $p_{i+1}=x_{j+1}\dots x_k x_1\dots x_l$ for some $j\in\{1,\dots,k-1\}$ and $l\in\{1,\dots, j-1\}$. But this yields the contradiction that $p_{i+1}$ is a subword of $p_i^2$ of length $<|p_i|$.\\
\\
Hence different solutions of (7) define different nod-paths in (6). The number of solutions of (7) is $\sim n^d$ and thus $n^d \preccurlyeq$ the number of nod-paths of length $n$ or less.
\end{proof}

\begin{remark}
Let $H$ be a nonfinite hypergraph. One can use Theorem \ref{thmgk} to determine $\GKdim L_K(H)$ as follows. Let $\{H_i\mid i\in I\}$ be the direct system of all finite subhypergraphs of $H$. Then $L_K(E,w)=\varinjlim\limits_i L_K(H_i)$ by Proposition \ref{proplim2}. By \cite[Theorem 3.1]{moreno-molina} we have $\GKdim L_K(H)=\sup\limits_i \GKdim L_K(H_i)$.
\end{remark}

In general it is not so easy to read off the quasi-cycles from a finite hypergraph. But there is the following algorithm to find all the quasi-cycles: For any vertex $v$ list all the d-paths $x_1\dots x_n$ starting and ending at $v$ and having the property that $x_i\neq x_j$ for any $i\neq j$ (there are only finitely many of them). Now delete from that list any $p$ such that $p^2$ is not a nod-path. Next delete from the list any $p$ such that $p^2$ has a subword $q$ of length $|q|<|p|$ such that $q^2$ is a nod-path. The remaining d-paths on the list are precisely the quasi-cycles starting (and ending) at $v$.

\begin{example}\label{ex11}
Consider again the hypergraph 
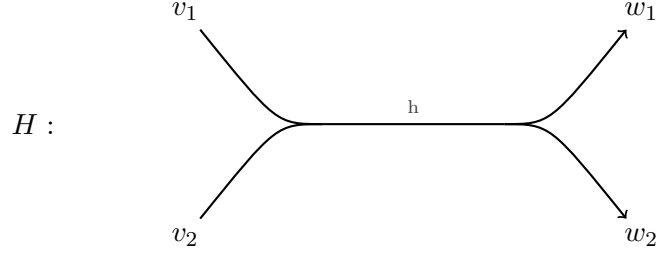
\begin{center}
\begin{tikzpicture}
  \node (H) at (0,-1.5) { $H:$ };
  \node (v_1) at (2,0)   { $v_1$ };
  \node (v_2) at (2,-3) { $v_2$ };
  \node (w_1) at (8,0)  { $w_1$ };
  \node (w_2) at (8,-3)  { $w_2$ };
  \connectFour[
  @ratio=.5,
  @pos1=.7,
  @pos2=.7,
  @edge 3=->,
  @edge 4=->,
  @edge=thick
  ]{v_1}{v_2}{w_1}{w_2}{h}
\end{tikzpicture}
\end{center}
from Example \ref{ex1}. By applying the algorithm described in the paragraph right before this example we find that the only quasi-cycles are $q=h_{22}h_{22}^*$ and $q'=h_{22}^*h_{22}$. Clearly $q\approx q'$. It follows from Theorem \ref{thmgk} that $\GKdim L_K(H)=1$.
\end{example}

\section{Valuations and local valuations}
\subsection{General results}

\begin{definition}\label{defval}
Let $R$ be a ring. A {\it valuation} on $R$ is a map $\nu:R\rightarrow \N_0\cup\{-\infty\}$ such that 
\begin{enumerate}[(i)]
\item $\nu(x)=-\infty\Leftrightarrow x=0$ for any $x\in R$,
\item $\nu(x-y)\leq \max\{\nu(x),\nu(y)\}$ for any $x,y\in R$ and
\item $\nu(xy)=\nu(x)+\nu(y)$ for any $x,y\in R$.
\end{enumerate}
We use the conventions $-\infty<n$ for any $n\in\N_0$ and $(-\infty)+n=n+(-\infty)=-\infty$ for any $n\in\N_0\cup\{-\infty\}$.
\end{definition}

\begin{remark}
One checks easily that condition (ii) in Definition \ref{defval} is satisfied iff the conditions (iia) and (iib) below are satisfied.
\begin{enumerate}[({ii}a)]
\item $\nu(x+y)\leq \max\{\nu(x),\nu(y)\}$ for any $x,y\in R$.
\item $\nu(x)=\nu(-x)$ for any $x\in R$.
\end{enumerate}
\end{remark}

Recall that a {\it domain} is a nonzero ring without zero divisors. 
\begin{lemma}\label{lemval}
Let $R$ be a nonzero ring that has a valuation. Then $R$ is a domain.
\end{lemma}
\begin{proof}
Let $\nu$ be a valuation on $R$. Let $x,y\in R$ such that $xy=0$. Then $-\infty=\nu(0)=\nu(xy)=\nu(x)+\nu(y)$ and hence $\nu(x)=-\infty$ or $\nu(y)=-\infty$. Thus $x=0$ or $y=0$.
\end{proof}

\begin{definition}
A ring with enough idempotents is a pair $(R,E)$ where $R$ is a ring and $E$ is a set
of nonzero orthogonal idempotents in $R$ for which the set of finite sums of
distinct elements of $E$ is a set of local units for $R$. Note that if $(R,E)$ is a ring with enough idempotents, then $R=\bigoplus\limits_{e\in E}eR=\bigoplus\limits_{f\in E}Rf=\bigoplus\limits_{e,f\in E}eRf$ as additive groups. A ring with enough idempotents $(R,E)$ is called {\it connected} if $eRf\neq\{0\}$ for any $e,f\in E$.
\end{definition}

\begin{definition}\label{deflocval}
Let $(R,E)$ be a ring with enough idempotents. A {\it local valuation} on $(R,E)$ is a map $\nu:R\rightarrow \N_0\cup\{-\infty\}$ such that 
\begin{enumerate}[(i)]
\item $\nu(x)=-\infty\Leftrightarrow x=0$ for any $x\in R$,
\item $\nu(x-y)\leq \max\{\nu(x),\nu(y)\}$ for any $x,y\in R$ and
\item $\nu(xy)=\nu(x)+\nu(y)$ for any $e\in E$, $x\in Re$ and $y\in eR$.
\end{enumerate}
A local valuation $\nu$ on $(R,E)$ is called {\it trivial} if $\nu(x)=0$ for any $x\in R\setminus\{0\}$ and {\it nontrivial} otherwise.
\end{definition}

Let $R$ be a ring. Recall that a left ideal $I$ of $R$ is called {\it essential} if $I \cap J=\{0\}~\Rightarrow ~J=\{0\}$ for any left ideal $J$ of $R$. If $I$ is an essential left ideal of $R$, then we write $I\subseteq_e R$. For any $x\in R$ define the left ideal $\ann(x):=\{y\in R \mid yx=0\}$. The ring $R$ is called {\it left nonsingular}, if $\ann(x)\subseteq_e R~\Leftrightarrow x=0$ for any $x\in R$. A right nonsingular ring is defined similarly. $R$ is called {\it nonsingular} if it is left and right nonsingular.

\begin{proposition}\label{proplocval1}
Let $(R,E)$ be a ring with enough idempotents that has a local valuation. Then $R$ is nonsingular. 
\end{proposition}
\begin{proof}
We show only left singularity of $R$ and leave the right singularity to the reader. Let $\nu$ be a local valuation on $(R,E)$ and $x\in R\setminus\{0\}$. Choose an $e\in E$ such that $ex\neq 0$. Suppose that $ye\in \ann(x)$ for some $y\in R$. Then 
\[\nu(ye)+\nu(ex)=\nu(yex)=\nu(0)=-\infty\]
and hence $ye=0$. This shows that $\ann(x)\cap Re=\{0\}$. But $Re\neq \{0\}$ since $e\in Re$. Hence $\ann(x)$ is not essential.
\end{proof}

Recall that a nonzero ring $R$ is called a {\it prime ring} if $IJ=\{0\}~\Rightarrow I=\{0\} \lor J=\{0\}$ for any ideals $I$ and $J$ of $R$. Equivalently, $R$ is a prime ring if $xRy=\{0\}~\Rightarrow~(x=0 \lor y=0)$ for any $x,y\in R$.

\begin{proposition}\label{proplocval2}
Let $(R,E)$ be a nonzero, connected ring with enough idempotents that has a local valuation. Then $R$ is a prime ring. 
\end{proposition}
\begin{proof}
Let $\nu$ be a local valuation on $(R,E)$ and $x,y\in R\setminus\{0\}$. Clearly there are $e,f\in E$ such that $xe, fy\neq 0$. Since $(R,E)$ is connected, we can choose a $z\in eRf\setminus\{0\}$. Clearly 
\[\nu(xzy)=\nu(xezfy)=\nu(xe)+\nu(z)+\nu(fy)\geq 0.\]
Hence $xzy\neq 0$ and thus $xRy\neq \{0\}$.
\end{proof}

Recall that a ring $R$ is called {\it von Neumann regular} if for any $x\in R$ there is a $y\in R$ such that $xyx=x$. 

\begin{proposition}\label{proplocval3}
Let $(R,E)$ be a ring with enough idempotents that has a nontrivial local valuation. Then $R$ is not von Neumann regular. 
\end{proposition}
\begin{proof}
Let $\nu$ be a nontrivial local valuation on $(R,E)$. Choose an $x\in R$ such that $\nu(x)>0$. By condition (ii) in Definition \ref{deflocval} we may assume that $x\in eRf$ for some $e,f\in E$. Let $y\in R$. Clearly 
\[\nu(xyx)=\nu(xfyex)=\nu(x)+\nu(fye)+\nu(x).\]
It follows that either $\nu(xyx)=-\infty$ (if $\nu(fye)=-\infty$) or $\nu(xyx)\geq 2\nu(x)$ (if $\nu(fye)\geq 0$). Thus $xyx\neq x$.
\end{proof}

Recall that a ring is called {\it semiprimitive} if its Jacobson radical is the zero ideal. 
\begin{proposition}\label{proplocval4}
Let $(R,E)$ be a connected ring with enough idempotents. Suppose $R$ is a $K$-algebra and there is a local valuation $\nu$ on $(R,E)$ such that $\nu(x)=0$ iff $x\in \Span(E)\setminus\{0\}$ where $\Span(E)$ denotes the linear subspace of $R$ spanned by $E$. Then $R$ is semiprimitive.
\end{proposition}
\begin{proof}
Let $\nu$ be the local valuation on $(R,E)$ such that $\nu(x)=0$ iff $x\in \Span(E)\setminus\{0\}$. Assume that the Jacobson radical $J$ of $R$ is not zero. Since $R=\bigoplus\limits_{v,w\in E}eRf$, there are $e,f\in E$ and an $x'\in J\cap eRf\setminus\{0\}$. Since $R$ is connected, we can choose an element $z\in fRe\setminus\{0\}$. Then $x:=x'z\in J\cap eRe\setminus\{0\}$ since $\nu(x)=\nu(x'z)=\nu(x')+\nu(z)\geq 0$. Since $J$ does not contain any nonzero idempotents, it follows that $\nu(x)>0$ (if $\nu(x)=0$, then $x=ke$ for some $k\in K$ and hence $J$ contains the nonzero idempotent $e$). Since $x\in J$, we have that $x$ is left quasi-regular, i.e. there is a $y\in R$ such that $x+y=yx$. By multiplying $e$ from the right and from the left one gets $x+eye=eyex$. Hence we may assume that $y\in eRe$. It follows that 
\[\max\{\nu(x),\nu(y)\}\geq \nu(x+y)=\nu(yx)=\nu(x)+\nu(y).\]
This implies that $\nu(y)=0$ and hence $y=ke$ for some $k\in K$. It follows that $y=yx-x=kex-x=kx-x=(k-1)x$. But this yield a contradiction since $\nu(y)=0$ but either $\nu((k-1)x)=-\infty$, if $k=1$, or $\nu((k-1)x)=\nu(x)>0$, if $k\neq 1$ (note that $\nu((k-1)x)=\nu((k-1)ex)=\nu((k-1)e)+\nu(x)$). Thus the Jacobson radical of $R$ is zero.
\end{proof}

\subsection{Applications to Leavitt path algebras of hypergraphs}

In this subsection $H$ denotes a hypergraph. Note that $(L_K(H),H^0)$ is a ring with enough idempotents. Set $X:=\{v,h_{ij}, h_{ij}^*\mid v\in H^0, h\in H^1, 1\leq i\leq |s(h)|, 1\leq j\leq |r(h)|\}$. Let $K\X$ be the free algebra generated by $X$ and $K\X_{\nod}$ the subspace of $K\X$ generated by the nod-paths. By Theorem \ref{thmbasis} there is an isomorphism $\NF:L_K(H)\rightarrow K\X_{\nod}$ of $K$-vector spaces. Note that for an $a\in L_K(H)$, its normal representative $\NF(a)$ can be obtained by taking an arbitrary representative of $a$ in $K\X$ and then applying reductions corresponding to the relations (i')-(v') in the proof of Theorem \ref{thmbasis} (i.e. replacing an occurrence of a word in the representative of $a$ that equals a LHS of one the relations (i')-(v') by the correponding RHS) in an arbitrary order until it is no longer possible.  For an $a\in L_K(H)$ we define its {\it support} $\supp(a)$ as the set of all nod-paths which appear in $\NF(a)$ with nonzero coefficient. Recall that if $p=x_1\dots x_n$ is a nod-path, then its length $|p|$ is $n$ if $x_1,\dots,x_n\in X\setminus H^0$ and $0$ if $n=1$ and $x_1\in H^0$. 

\begin{definition}
$H$ satisfies {\it Condition (LV)} if $|s(h)|, |r(h)|\geq 2$ for any $h\in H^1$.
\end{definition}

\begin{theorem}\label{thmlocval}
If $H$ satisfies Condition (LV), then the map
\begin{align*}
\nu:L_K(H)&\rightarrow \N_0\cup\{-\infty\}\\
a&\mapsto \max\{|p|\ | \ p\in \supp(a)\}.
\end{align*}
is a local valuation on $(L_K(H), H^0)$. Here we use the convention $\max(\emptyset)=-\infty$.
\end{theorem}
\begin{proof}
Obviously conditions (i) and (ii) in Definition \ref{deflocval} are satisfied. It remains to show that condition (iii) is satisfied. Let $v\in H^0$, $a\in L_K(H)v$ and $b\in vL_K(H)$. If one of the terms $\nu(a)$ and $\nu(b)$ equals $0$ or $-\infty$, then clearly  $\nu(ab)=\nu(a)+\nu(b)$. Suppose now $\nu(a),\nu(b)\geq 1$. Clearly $\nu(ab)\leq \nu(a)+\nu(b)$ since a reduction preserves or decreases the length of a monomial. It remains to show that $\nu(ab)\geq \nu(a)+\nu(b)$. Let 
\[p_k=x^k_1\dots x^k_{\nu(a)}~(1\leq k \leq r)\]
be the elements of $\supp(a)$ with maximal length (namely $\nu(a)$) and 
\[q_l=y^l_1\dots y^l_{\nu(b)}~(1\leq l \leq s)\]
be the elements of $\supp(b)$ with maximal length (namely $\nu(b)$). We assume that the $p_k$'s are pairwise distinct and also that the $q_l$'s are pairwise distinct. Note that $x^k_i, y^l_j\in X\setminus H^0$ for any $i,j,k,l$ since $\nu(a),\nu(b)\geq 1$. Since $\NF$ is a linear map, we have 
\begin{displaymath}
\NF(p_k q_l)= \left\{
\begin{array}{ll}
p_k q_l, & \text{if }x^k_{\nu(a)}y^l_1\text{ is not forbidden},\\
\\
\NF(\delta_{ij}x^k_1\dots x^k_{\nu(a)-1}s(h)_iy^l_2\dots y^l_{\nu(b)})\\-\sum\limits_{k=2}^{|r(h)|}x^k_1\dots x^k_{\nu(a)-1}h_{ik}h_{jk}^*y^l_2\dots y^l_{\nu(b)}, & \text{if }x^k_{\nu(a)}y^l_1=h_{i1}h_{j1}^*,\\
\\
\NF(\delta_{ij}x^k_1\dots x^k_{\nu(a)-1}r(h)_iy^l_2\dots y^l_{\nu(b)})\\-\sum\limits_{k=2}^{|s(h)|}x^k_1\dots x^k_{\nu(a)-1}h_{ki}^*h_{kj}y^l_2\dots y^l_{\nu(b)}, & \text{if }x^k_{\nu(a)}y^l_1=h_{1i}^*h_{1j}.
\end{array}\right.
\end{displaymath}

\medskip

\begin{enumerate}

\item[Case 1]   {\it Assume that $x^k_{\nu(a)}y^l_1$ is not forbidden for any $k,l$}.\\
Then $p_k q_l\in \supp(ab)$ for any $k,l$. It follows that $\nu(ab)\geq |p_k q_l|=\nu(a)+\nu(b)$.
 
\medskip 

\item[Case 2] {\it Assume that there are $k,l$ such that $x^k_{\nu(a)}y^l_1=h_{i1}h_{j1}^*$ for some $h\in H^1$ and $1\leq i,j\leq |s(h)|$.}

\medskip

\begin{enumerate}

\item [Case 2.1] {\it Assume $p_{k'} q_{l'}\neq x^k_1\dots x^k_{\nu(a)-1}h_{i2}h_{j2}^*y^l_2\dots y^l_{\nu(b)}$ for any $k',l'$.}\\
Then
\[x^k_1\dots x^k_{\nu(a)-1}h_{i2}h_{j2}^*y^l_2\dots y^l_{\nu(b)}\in \supp(ab)\]
since it does not cancel with another term. It follows that $\nu(ab)\geq \nu(a)+\nu(b)$. 

\medskip 

\item[Case 2.2]  {\it Assume $p_{k'}  q_{l'}= x^k_1\dots x^k_{\nu(a)-1}h_{i2}h_{j2}^*y^l_2\dots y^l_{\nu(b)}$ for some $k',l'$.}\\
One checks easily that in this case 
\[p_k  q_{l'}=x^k_1\dots x^k_{\nu(a)-1}h_{i1}h_{j2}^*y^l_2\dots y^l_{\nu(b)}\in \supp(ab).\]
It follows that $\nu(ab)\geq \nu(a)+\nu(b)$.

\end{enumerate}

\medskip 
\item[Case 3] {\it Assume that there are $k,l$ such that $x^k_{\nu(a)}y^l_1=h_{1i}^*h_{1j}$ for some $h\in h^1$ and $1\leq i,j\leq |r(h)|$.}

\medskip

\begin{enumerate}

\item [Case 3.1] {\it Assume $p_{k'} q_{l'}\neq x^k_1\dots x^k_{\nu(a)-1}h_{2i}^*h_{2j}y^l_2\dots y^l_{\nu(b)}$ for any $k',l'$. }\\
Then
\[x^k_1\dots x^k_{\nu(a)-1}h_{2i}^*h_{2j}y^l_2\dots y^l_{\nu(b)}\in \supp(ab)\]
since it does not cancel with another term. It follows that $\nu(ab)\geq \nu(a)+\nu(b)$. 

\medskip 

\item[Case 3.2]  {\it Assume $p_{k'}  q_{l'}= x^k_1\dots x^k_{\nu(a)-1}h_{2i}^*h_{2j}y^l_2\dots y^l_{\nu(b)}$ for some $k',l'$.}\\
One checks easily that in this case 
\[p_k  q_{l'}=x^k_1\dots x^k_{\nu(a)-1}h_{1i}^*h_{2j}y^l_2\dots y^l_{\nu(b)}\in \supp(ab)\]
It follows that $\nu(ab)\geq \nu(a)+\nu(b)$.

\end{enumerate}
\end{enumerate}
Hence condition (iii) also is satisfied and thus $\nu$ is a local valuation on $(L_K(H),H^0)$.
\end{proof}

\begin{corollary}\label{corlocval1}
Suppose that $H$ satisfies Condition (LV). Then $L_K(H)$ is nonsingular. 
\end{corollary}
\begin{proof}
Follows from Proposition \ref{proplocval1} and Theorem \ref{thmlocval}.
\end{proof}

Recall that a d-path is a path in the double graph $\hat E$ of the directed graph $E$ associated to $H$. We call $H$ {\it connected}, if for any $v,w\in H^0$ there is a d-path $p$ such that $s(p)=v$ and $r(p)=w$. 
\begin{corollary}\label{corlocval2}
Suppose that $H$ satisfies Condition (LV) and is connected. Then $L_K(H)$ is a prime ring. 
\end{corollary}
\begin{proof}
Cleary $L_K(H)$ is not the zero ring since $H$ is not the empty hypergraph. Let $v, w\in H^0$. Since $H$ is connected, there is a d-path $p$ such that $s(p)=v$ and $r(p)=w$. Let $\nu$ be the local valuation on $(L_K(H),H^0)$ defined in Theorem \ref{thmlocval}. Clearly $\nu(p)=|p|\geq 0$ since $\nu(u)=0$ for any $u\in H^0$ and $\nu(h_{ij})=\nu(h_{ij}^*)=1$ for any $h\in H^1$, $1\leq i\leq |s(h)|$ and $1\leq j\leq |r(h)|$. Hence $vL_K(H)w\neq \{0\}$ for any $v, w\in H^0$ and therefore $(L_K(H),H^0)$ is a connected ring with enough idempotents. It follows from Proposition \ref{proplocval2} that $L_K(H)$ is a prime ring.
\end{proof}

\begin{corollary}\label{corlocval3}
Suppose that $H$ satisfies Condition (LV) and that $H^1\neq\emptyset$. Then $L_K(H)$ is not von Neumann regular. 
\end{corollary}
\begin{proof}
Let $\nu$ be the local valuation on $(L_K(H),H^0)$ defined in Theorem \ref{thmlocval}. Choose an $h\in H^1$. Then $\nu(h_{11})=1$ and therefore $\nu$ is nontrivial. It follows from Proposition \ref{proplocval3} that $L_K(H)$ is not von Neumann regular.
\end{proof}

\begin{corollary}\label{corlocval4}
Suppose that $H$ satisfies Condition (LV) and is connected. Then $L_K(H)$ is semiprimitive. 
\end{corollary}
\begin{proof}
Let $\nu$ be the local valuation on $(L_K(H),H^0)$ defined in Theorem \ref{thmlocval}. Then clearly $\nu(x)=0$ iff $x\in \Span(H^0)\setminus\{0\}$ where $\Span(H^0)$ denotes the linear subspace of $L_K(H)$ spanned by $H^0$. Moreover, $(L_K(H),H^0)$ is connected since $H$ is connected (see the proof of Corollary \ref{corlocval2}). It follows from Proposition \ref{proplocval4} that $L_K(H)$ semiprimitive.
\end{proof}

\begin{example}\label{ex111}
Consider again the hypergraph 
\begin{center}
\begin{tikzpicture}
  \node (H) at (0,-1.5) { $H:$ };
  \node (v_1) at (2,0)   { $v_1$ };
  \node (v_2) at (2,-3) { $v_2$ };
  \node (w_1) at (8,0)  { $w_1$ };
  \node (w_2) at (8,-3)  { $w_2$ };
  \connectFour[
  @ratio=.5,
  @pos1=.7,
  @pos2=.7,
  @edge 3=->,
  @edge 4=->,
  @edge=thick
  ]{v_1}{v_2}{w_1}{w_2}{h}
\end{tikzpicture}
\end{center}
from Examples \ref{ex1} and \ref{ex11}. As shown in Example \ref{ex11}, $\GKdim L_K(H)=1$. Clearly $H$ satisfies Condition (LV) and is connected. It follows from the Corollaries \ref{corlocval1}, \ref{corlocval2}, \ref{corlocval3} and \ref{corlocval4} that $L_K(H)$ is nonsingular, prime, not von Neumann regular and semiprimitive. Since $L_K(H)$ is a finitely generated prime algebra of GK dimension one, $L_K(H)$ is fully bounded Noetherian and finitely generated as a module over its center (see \cite{small-warfield}). It follows that $L_K(H)$ is a PI-ring.
\end{example}


\subsection{Classification of the hypergraphs $H$ such that $L_K(H)$ is a domain}

\begin{definition}
A hypergraph $H$ satisfies {\it Condition (B)} if any $h\in H^1$ has the property that either $|s(h)|, |r(h)|\geq 2$ or $|s(h)|=|r(h)|=1$.
\end{definition}

\begin{theorem}
Let $H$ be a hypergraph, then $L_K(H)$ is a domain iff $|H^0|=1$ and $H$ satisfies Condition (B).
\end{theorem}
\begin{proof}
If there are distinct $v,w\in H^0$. Then $vw=0$ in $L_K(H)$ by relation (i) in Definition \ref{defhlpa}. But $v,w\neq 0$ in $L_K(H)$ by Theorem \ref{thmbasis} (note that $v$ and $w$ are nod-paths). Hence $L_K(H)$ is not a domain.\\
Suppose now that Condition (B) is not satisfied. Then there is an $h\in H^1$ such that $|s(h)|=1~\land~|r(h)|\geq 2$ or $|s(h)|\geq 2~\land~|r(h)|=1$. We only consider the first case and leave the second to the reader. Choose $i\neq j\in \{1,\dots,|r(h)|\}$ (possible since $|r(h)|\geq 2$). Then $h_{1i}^*h_{1j}=0$ by relation (iv) in Definition \ref{defhlpa}. But $h_{1i}^*,h_{1j}\neq 0$ in $L_K(H)$ by Theorem \ref{thmbasis} (note that $h_{1i}^*$ and $h_{1j}$ are nod-paths). Hence $L_K(H)$ is not a domain.\\
Now suppose that $|H^0|=1$ and $H$ satisfies Condition (B). For any $h\in H^1$ define a hypergraph $H_h=(H_h^0,H_h^1, s_h, r_h)$ by $H_h^0=H^0$, $H_h^1=\{h\}$, $s_h(h)=s(h)$ and $r_h(h)=r(h)$. If $|s(h)|, |r(h)|\geq 2$, then $H_h$ satisfies Condition (LV) and hence $L_K(H_h)$ is a domain by Lemma \ref{lemval} and Theorem \ref{thmlocval} (the local valuation $\nu$ is in fact a valuation since $|H^0|=1$). If $|s(h)|=|r(h)|=1$, then $L_K(H_h)$ is isomorphic to the Laurent polynomial ring $K[X,X^{-1}]$, which is a domain. Hence all the algebras $L_K(H_h)~(h\in H^1)$ are domains. One checks easily that $L_K(H)$ is the coproduct (sometimes also called ``free product'') of the $K$-algebras $L_K(H_h)~(h\in H^1)$. It follows from \cite[Theorem 3.2]{cohn68} that $L_K(H)$ is a domain (note that a ``$1$-fir'' is the same as a domain).
\end{proof}

\section{Additional applications of the linear bases}
\subsection{General results}
In this subsection $(A,E)$ denotes a {\it $K$-algebra with enough idempotents}, i.e. $A$ is a $K$-algebra and $(A,E)$ is a ring with idempotents. We call an element $a\in A$ {\it homogeneous} if $a\in eAf$ for some $e,f\in E$ (recall that $A=\bigoplus\limits_{e,f\in E}eAf$). $B$ denotes a $K$-basis for $A$ which consists of homogeneous elements and contains $E$. Moreover, $\mu$ denotes a map $B\rightarrow \N_0$ which has the property that $\mu(b)=0~\Leftrightarrow~b\in E$.

\begin{definition}
An element $b\in B\cap eAf$ is called {\it left adhesive} if $ab \in B$ for any $a\in B \cap Ae$ and {\it right adhesive} if $bc \in B$ for any $c\in B\cap fA$. An element $b\in B$ is called {\it adhesive} if it is left and right adhesive. A {\it left valuative basis element} is a left adhesive element $b\in B\cap eA$ such that $\mu(ab)=\mu(a)+\mu(b)$ for any $a\in B \cap Ae$. A {\it right valuative basis element} is a right adhesive element $b\in B\cap Af$ such that $\mu(bc)=\mu(b)+\mu(c)$ for any $c\in B\cap fA$. A {\it valuative basis element} is an adhesive element $b\in B\cap eAf$ such that $\mu(abc)=\mu(a)+\mu(b)+\mu(c)$ for any $a\in B \cap Ae$ and $c\in B\cap fA$.
\end{definition}

If $S\in A$ is a subset, then we denote by $I_l(S)$ the left ideal of $A$ generated by $S$, by $I_r(S)$ the right ideal of $A$ generated by $S$ and by $I(S)$ the ideal of $A$ generated by $S$. The following lemma is straightforward to check.

\begin{lemma}\label{lemidbas}
The following is true:
\begin{enumerate}[(i)]
\item If $S_l=\{b_i\in B\cap e_iAf_i\mid i\in \Phi\}$ is a set of left adhesive basis elements, then $I_l(S_l)$ is free with basis $\{ab_i\mid i\in \Phi,a\in B\cap Ae_i\}$. 
\item If $S_r=\{b_i\in B\cap e_iAf_i\mid i\in \Phi\}$ is a set of right adhesive basis elements, then $I_r(S_r)$ is free with basis $\{b_ic\mid i\in\Phi, c\in B\cap f_iA\}$.
\item If $S=\{b_i\in B\cap e_iAf_i\mid i\in \Phi\}$ is a set of adhesive basis elements, then $I(S)$ is free with basis $\{ab_ic\mid i\in\Phi,a\in B\cap Ae_i, c\in B\cap f_iA\}$.
\end{enumerate}
\end{lemma}

\begin{lemma}\label{lemvalbas}
Distinct left valuative basis elements generate distinct left ideals, distinct right valuative basis elements generate distinct right ideals and distinct valuative basis elements generate distinct ideals.
\end{lemma}
\begin{proof}
Let $b\in B\cap eAf$ and $b'\in B\cap e'Af'$ be valuative basis elements. W.l.o.g assume that $\mu(b)\geq\mu(b')$. Assume that $b'\in I(b)$. By the previous lemma, $B_{I(b)}:=\{abc\mid a\in B\cap Ae, c\in B\cap fA\}$ is a basis for $I(b)$. Hence $b'=\sum\limits_{x\in B_{I(b)}}k_xx$ where almost all $k_x$ are zero. Since $b$ is valuative, $\mu(x)>\mu(b)\geq \mu(b')$ for any $x\in B_{I(b)}\setminus\{b\}$. It follows that $k_x=0$ for any $x\in B_{I(b)}\setminus\{b\}$ and $k_b=1$. Thus $b=b'$. The other assertions of the lemma can be shown similarly.
\end{proof}

\begin{proposition}\label{propvalbas1}
Suppose there exists a valuative basis element $b\in (B\setminus E)\cap eAe$. Then $\dim_K(A)=\infty$, $A$ is not simple (in fact it has infinitely many ideals), neither left nor right Artinian and not von Neumann regular.
\end{proposition}
\begin{proof}
Clear $\mu(b^n)=n\mu(b)$ since $b$ is valuative. Hence $b^m\neq b^n$ if $m\neq n$ (note that $\mu(b)\geq 1$ since $b\not\in E$). Therefore $\dim_K(A)=\infty$. Lemma \ref{lemvalbas} implies that we have an infinite descending chain $I(b)\supsetneq I(b^2)\supsetneq I(b^3)\dots$ of ideals (note that $b^n$ is valuative for any $n\in\N$ since $b$ valuative). Hence $A$ is neither left nor right Artinian. Now assume that $bab=b$ for some $a\in A$. Clearly we may assume that $a\in eAe$. Write $a=\sum\limits_{c\in B\cap eAe}k_cc$ where almost all $k_c$ are zero. Then $b=bab=\sum\limits_{c\in B\cap eAe}k_cbcb$. But $\mu(bcb)=\mu(b)+\mu(c)+\mu(b)>\mu(b)$ for any $c\in B\cap eAe$ and hence we arrived at a contradiction. Thus $A$ is not von Neumann regular.
\end{proof}

\begin{definition}
We say that two basis elements $b\in B\cap eAf$ and $b'\in B\cap e'Af'$ {\it have no common left multiple} if there are no $a\in B\cap Ae,a'\in B\cap Ae'$ such that $ab=a'b'$. We say that $b$ and $b'$ {\it have no common right multiple} if there are no $c\in B\cap fA,c'\in B\cap f'A$ such that $bc=b'c'$.
\end{definition}

\begin{definition}
An element $b\in B\cap eAf$ is called {\it right cancellative} if $ab=cb~\Rightarrow~a=c$ for any $a,c \in B\cap Ae$ and {\it left cancellative} if $ba'=bc'~\Rightarrow~a'=c'$ for any $a',c' \in B\cap fA$.
\end{definition}

\begin{proposition}\label{propvalbas2}
If there are elements $b,b'\in B\cap eAe$ such that $b$ is adhesive and right cancellative, $b'$ is left adhesive and $b$ and $b'$ have no common left multiple, then $A$ is not left Noetherian. If there are elements $c,c'\in B\cap eAe$ such that $c$ is adhesive and left cancellative, $c'$ is right adhesive and $c$ and $c'$ have no common right multiple, then $A$ is not right Noetherian.
\end{proposition}
\begin{proof}
For any $n\in \N$ set $I_n:=I_l(b'b,b'b^2, \dots, b'b^n)$. Then clearly $I_n\subseteq I_{n+1}$ for any $n\in \N$. Now assume that $b'b^{n+1}\in I_n$ for some $n\in \N$. By Lemma \ref{lemidbas} the set $\{ab'b^i\mid 1\leq i\leq n, a\in B\cap Ae \}$ is a basis for $I_n$ (note that $b'b^i$ is adhesive for any $1\leq i \leq n$ since $b$ is adhesive and $b'$ is left adhesive). It follows that $b'b^{n+1}=ab'b^i$ for some $1\leq i\leq n$ and $a\in B\cap Ae$. This implies $b'b^{n+1-i}=ab'$ since $b$ is right cancellative. But this contradicts the assumption that $b$ and $b'$ have no common left multiple. Hence $I_n\subsetneq I_{n+1}$ for any $n\in \N$ and therefore $A$ is not left Noetherian. Similarly one can prove the second assertion of the proposition.
\end{proof}

\subsection{Applications to Leavitt path algebras of hypergraphs}

\begin{definition}
A hypergraph $H$ satisfies {\it Condition (A)} if there is an $h\in H^1$ such that $|s(h)|, |r(h)|\geq 2$.
\end{definition}
\begin{remark}$~$\\
\vspace{-0.6cm}
\begin{enumerate}[(a)]
\item If $H$ is a hypergraph that does not satisfy Condition (A), then $L_K(H)$ is isomorphic to a Leavitt path algebra of a finitely separated graph, cf. Example \ref{ex3}. 
\item If $H$ is a hypergraph, then 
\[\text{Condition (A')}~\Longrightarrow~\text{Condition (A)}~\overset{H^1\neq \emptyset}{\Longleftarrow}~\text{Condition (LV)}~\Longrightarrow~\text{Condition (B)}.\]
\end{enumerate}
\end{remark}

The theorem below shows that a Leavitt path algebra $L_K(H)$ that is finite-dimensional as a $K$-vector space or simple or left Artinian or right Artininan or von Neumann regular is isomorphic to a Leavitt path algebra of a finitely separated graph.

\begin{theorem}
Let $H$ be a hypergraph that satisfies Condition (A). Then $\dim_K(L_K(H))=\infty$, $L_K(H)$ is not simple (in fact it has infinitely many ideals), neither left nor right Artinian and not von Neumann regular.
\end{theorem}
\begin{proof}
Let $B$ be the basis for $L_K(H)$ consisting of all nod-paths. Clearly $B$ consists of elements that are homogeneous with respect to $(L_K(H), H^0)$ and contains $H^0$. Define a map $\mu:B\rightarrow \N_0$ by $\mu(p)=|p|$ for any nod-path $p$. Then clearly $\mu(b)=0$ iff $b\in H^0$. Since $H$ satisfies Condition (A), we can choose an $h\in H^1$ such that $|s(h)|, |r(h)|\geq 2$. Set $b:=h_{22}h_{22}^*$. One checks easily that $b\in(B\setminus H^0)\cap s(h)_2L_K(H)s(h)_2$ is a valuative basis elements (note that $h_{22}$ and $h_{22}^*$ do not appear in a forbidden word, see Definition \ref{defnod}). The assertion of the theorem follows now from Proposition \ref{propvalbas1}.
\end{proof}

Example \ref{ex111} shows that there is a hypergraph $H$ which satisfies Condition (A) (even Condition (LV)) such that $L_K(H)$ is Noetherian. But the next theorem shows that if a hypergraph $H$ satisfies Condition (A'), then $L_K(H)$ cannot be left or right Noetherian.

\begin{theorem}
Let $H$ be a hypergraph that satisfies Condition (A'). Then $L_K(H)$ is neither left nor right Noetherian.
\end{theorem}
\begin{proof}
Define $B$ and $\mu$ as in the proof of the previous theorem. Since $H$ satisfies Condition (A'), we can choose an $h\in H^1$ such that $|s(h)|, |r(h)|\geq 2$ and $s(h)$ is a proper multiset or $r(h)$ is a proper multiset or $s(h)$ and $r(h)$ are sets with nontrivial intersection.\\
First suppose that $s(h)$ is a proper multiset. Choose $1\leq i<j\leq |s(h)|$ such that $s(h)_i=s(h)_j$. Set $b:=h_{j2}h_{j2}^*$, $b':=h_{j2}h_{i2}^*$ and $b'':=h_{i2}h_{j2}^*$. Then clearly $b,b',b''\in(B\setminus H^0)\cap vL_K(H)v$ where $v=s(h)_i=s(h)_j$. One checks easily that $b$ is adhesive and both left and right cancellative, $b'$ is left adhesive, $b''$ is right adhesive, $b$ and $b'$ have no common left multiple and $b$ and $b''$ have no common right multiple (note that $h_{j2}$ and $h_{j2}^*$ do not appear in a forbidden word, see Definition \ref{defnod}). It follows from Proposition \ref{propvalbas2} that $L_K(H)$ is neither left nor right Noetherian. The case that $r(h)$ is a proper multiset can be handled similarly. \\
Now suppose that $s(h)$ and $r(h)$ are sets with nontrivial intersection. By Remark \ref{remhlpa} (a) we may assume that $s(h)_2=r(h)_2$. Set $b:=h_{22}$ and $b':=h_{22}^*$. Then clearly $b,b'\in(B\setminus H^0)\cap vL_K(H)v$ where $v=s(h)_2=r(h)_2$. One checks easily that $b$ and $b'$ are adhesive and both left and right cancellative and that $b$ and $b'$ have neither a common left multiple nor a common right multiple (note that $h_{22}$ and $h_{22}^*$ do not appear in a forbidden word, see Definition \ref{defnod}). It follows from Proposition \ref{propvalbas2} that $L_K(H)$ is neither left nor right Noetherian.
\end{proof}

\section{The $V$-monoid}
Let $R$ be a ring. Recall that a left $R$-module $M$ is called {\it unital} if $RM=M$. We denote by $R$-$\Mod$ the category of unital left $R$-modules. Furthermore we denote by $R$-$\Mod_{\proj}$ the full subcategory of $R$-$\Mod$ whose objects are the projective objects of $R$-$\Mod$ that are finitely generated as a left $R$-module. When $R$ has local units, then
\[V(R)=\{[P]\mid P\in R\text{-}\Mod_{\proj}\},\]
see e.g. \cite[Subsection 4A]{ara-hazrat-li-sims}. $V(R)$ becomes an abelian monoid by defining $[P]+[Q]=[P\oplus Q]$. 

Recall from Section 4 that $\Hy$ denotes the category of hypergraphs, $\A$ the category of $K$-algebras with local units and that $L_K:\Hy \rightarrow \A$ is a functor that commutes with direct limits. We denote the category of abelian monoids by $\M$. One checks easily that $V$ defines a continuous functor $\A\rightarrow \M$ in a canonical way. In Subsection 10.1 we define a functor $M:\Hy\rightarrow \M$. In Subsection 10.2 we recall some universal ring constructions by G. Bergman which will be used in the proof of the main result of this section, namely Theorem \ref{thmm}. In Subsection 10.3 we prove Theorem \ref{thmm} which states that $V\circ L_K\cong M$ and that $L_K(H)$ is left and right hereditary provided $H$ is finite. 

\subsection{The functor $M:\Hy\rightarrow \M$}

\begin{definition}\label{defM}
For any hypergraph $H$ let $M(H)$ be the abelian monoid presented by the generating set $H^0$ and the relations $\sum s(h)=\sum r(h)~(h\in H^1)$ where $\sum s(h)=\sum\limits_{i=1}^{|s(h)|}s(h)_i$ and $\sum r(h)=\sum\limits_{j=1}^{|r(h)|}r(h)_j$. If $\phi:H\rightarrow I$ is a morphism in $\Hy$, then there is a unique monoid homomorphism $M(\phi):M(H)\rightarrow M(I)$ such that $M(\phi)(v)=\phi^0(v)$ for any $v\in H^0$. One checks easily that $M:\Hy\rightarrow \M$ is a functor that commutes with direct limits.
\end{definition}

\subsection{Some universal ring constructions by G. Bergman}
In this subsection all rings are assumed to be unital. Let $k$ be a commutative ring and $R$ a $k$-algebra (i.e. $R$ is a ring given with a homomorphism of $k$ into its center). A {\it $R$-ring$_k$} is a $k$-algebra $S$ given with a $k$-algebra homomorphism $R\rightarrow S$. By an $R$-module we mean a right $R$-module. In \cite{bergman74}, G. Bergman described the following two key constructions: \\

\begin{itemize}
\item ADJOINING MAPS Let $M$ be any $R$-module and $P$ a finitely generated projective $R$-module. Then there exists $R$-ring$_k$ $S$, having a universal module homomorphism $f: M\otimes S \rightarrow P\otimes S$, see \cite[Theorem 3.1]{bergman74}. $S$ can be obtained by adjoining to $R$ a family of generators subject to certain relations, see \cite[Proof of Theorem 3.1]{bergman74}.\\
\item IMPOSING RELATIONS Let $M$ be any $R$-module, $P$ a projective $R$-module and $f: M\rightarrow P$ any module homomorphism. Then there exists an $R$-ring$_k$ $S$ such that $f\otimes S = 0$ and universal for that property. $S$ can be chosen to be a quotient of $R$, see \cite[Proof of Theorem 3.2]{bergman74}.
\end{itemize}

Using the key constructions above Bergman described more complicated constructions. One of them is used in this paper:\\

\begin{itemize}
\item ADJOINING ISOMORPHISMS Given two finitely generated projective $R$-modules $P$
and $Q$, one can adjoin a universal isomorphism between $P\otimes $ and $Q\otimes$ by first
freely adjoining maps $i:P\otimes \rightarrow Q \otimes$ and $i^{-1}:Q\otimes \rightarrow P \otimes$ (via ADJOINING MAPS) and then setting $1_{Q\otimes}-ii^{-1}$ and $1_{P\otimes}-i^{-1}i$ equal to $0$ (via IMPOSING RELATIONS), see \cite[p. 38]{bergman74}. Bergman denoted the resulting $R$-ring$_k$ by $R\langle i, i^{-1} : \overline{P}\cong\overline{Q}\rangle$.
\end{itemize}

Set $S:=R\langle i, i^{-1} : \overline{P}\cong\overline{Q}\rangle$. Bergman proved the following (for these results he required that $k$ is a field and that $P$ and $Q$ are nonzero): The abelian monoid $V(S)$ may be obtained from $V(R)$ by imposing one relation $[P] = [Q]$. Further the right global dimension of $S$ equals the right global dimension of $R$, unless the right global dimension of $R$ is $0$, in which case the right global dimension of $S$ is $\leq 1$. See \cite[Theorem 5.2 and last paragraph of p. 48]{bergman74}.

It is easily seen that Bergman's results mentioned above also apply to left $R$-modules.

\subsection{The monoid $V(L_K(H))$}
\begin{theorem}\label{thmm}
$V\circ L_K\cong M$. Moreover, if $H$ is a finite hypergraph, then $L_K(H)$ is left and right hereditary.
\end{theorem}
\begin{proof}
We have divided the proof into two parts, Part I and Part II. In Part I we define a natural transformation $\theta:M\rightarrow V\circ L_K$. In Part II we show that $\theta$ is a natural isomorphism and further that $L_K(H)$ is left and right hereditary provided that $H$ is finite.\\
\\  
{\bf Part I}
Let $H$ be a hypergraph and $F(H)$ the free abelian monoid generated by $H^0$. There is a unique monoid homomorphism $\eta_H:F(H)\rightarrow V(L_K(H))$ such that $\eta_H(v)=[L_K(H)v]$ for any $v\in H^0$. In order to show that $\eta_H$ induces a monoid homomorphism $M(H)\rightarrow V(L_K(H))$ we have to check that $\eta_H(\sum s(h))=\eta_H(\sum r(h))$ for any $h\in H^1$, i.e. that 
\begin{equation}
[\bigoplus\limits_{i=1}^{|s(h)|}L_K(H)s(h)_i]=[\bigoplus\limits_{j=1}^{|r(h)|}L_K(H)r(h)_j]\text{ for any }h\in H^1.
\end{equation}
Given $h\in H^1$ let $X_h$ be the $|s(h)|\times |r(h)|$-matrix whose entry at position $(i,j)$ is $h_{ij}$ and $Y_h$ the $|r(h)|\times |s(h)|$-matrix whose entry at position $(i,j)$ is $h_{ji}^*$. It follows from  relations (iii) and (iv) in Definition \ref{defhlpa} that
\begin{align*}
X_hY_h=\begin{pmatrix}s(h)_1&&\\&\ddots&\\&&s(h)_{|s(h)|}\end{pmatrix}\text{ and }Y_hX_h\begin{pmatrix}r(h)_1&&\\&\ddots&\\&&r(h)_{|r(h)|}\end{pmatrix}.
\end{align*}
Hence $X_h$ defines an isomorphism $\bigoplus\limits_{i=1}^{|s(h)|}L_K(H)s(h)_i\rightarrow \bigoplus\limits_{j=1}^{|r(h)|}L_K(H)r(h)_j$ by right multiplication (its inverse is defined by $Y_h$). Thus (8) holds and therefore $\eta_H$ induces a monoid homomorphism $\theta_H:M(H)\rightarrow V(L_K(H))$. It is an easy exercise to show that $\theta:M\rightarrow V\circ L_K$ is a natural transformation.\\
\\
{\bf Part II} We want to show that the natural transformation $\theta:M\rightarrow V\circ L_K$ defined in Part I is a natural isomorphism, i.e. that $\theta_{H}:M(H)\rightarrow V(L_K(H))$ is an isomorphism for any hypergraph $H$. By Proposition \ref{proplim} any hypergraph is a direct limit of a direct system of finite hypergraphs. Hence it is sufficient to show that $\theta_{H}$ is an isomorphism for any finite hypergraph $H$ (note that $M$, $V$ and $L_K$ commute with direct limits).\\
Let $H$ be a finite hypergraph. Set $A_0:=K^{H^0}$. We denote by $\alpha_v$ the element of $A_0$ whose $v$-component is $1$ and whose other components are $0$. Write $H^1=\{h^1,\dots,h^n\}$. We inductively define $K$-algebras $A_1,\dots,A_n$ as follows. Let $1\leq t \leq n$ and assume that $A_{t-1}$ has already been defined. Set $A_t:=A_{t-1}\langle i,i^{-1}:\overline{P_t}\cong\overline{Q_t}\rangle$ (see \cite[p. 38]{bergman74}) where
\[P_t=\bigoplus\limits_{j=1}^{|s(h^t)|}A_{t-1}\alpha_{s(h^t)_j}\text{  and  }Q_t=\bigoplus\limits_{k=1}^{|r(h^t)|}A_{t-1}\alpha_{r(h^t)_k}.\]
Investigating the proofs of Theorems 3.1 and 3.2 in \cite{bergman74} we see that $L_K(H)\cong A_n$. By \cite[Theorems 5.1, 5.2]{bergman74}, the abelian monoid $V(A_n)$ is isomorphic to $M(H)$. The monoid isomorphism $M(H)\rightarrow V(L_K(H))$ one gets in this way is precisely $\theta_{H}$. \\
Furthermore, the left global dimension of $A_n\cong L_K(E,w)$ is $\leq 1$ by \cite[Theorems 5.1, 5.2]{bergman74}, i.e. $L_K(H)$ is left hereditary. Since $L_K(H)$ is a ring with involution by Proposition \ref{propprop} (ii), we have $L_K(H)\cong L_K(H)^{op}$. Thus $L_K(H)$ is also right hereditary.
\end{proof}

Recall that a monoid $N$ is called {\it conical} if $m+n=0~\Rightarrow~m=n=0$ for any $m,n\in N$. It is easy to see that $V(L_K(H))\cong M(H)$ is conical for any hypergraph $H$. The proposition below shows that  conversely one can find for any conical abelian monoid $N$ a hypergraph $H$ such that  $N\cong V(L_K(H))$ (follows also from Example \ref{ex3} and \cite[Proposition 4.4]{aragoodearl}).
\begin{proposition}
For any conical abelian monoid $N$ there is a hypergraph $H$ such that $N\cong V(L_K(H))$.
\end{proposition}
\begin{proof}
Choose a presentation of $N$ with a nonempty set $\{x_j \mid j\in J\}$ of generators
and a nonempty set $\{r_i \mid i\in I\}$ of relations
\[r_i:~ \sum\limits_{j\in J} a_{ij}x_j=\sum\limits_{j\in J}b_{ij}x_j\]
where for each $i$ almost all but not all $a_{ij}$ are zero and similarly almost all but not all $b_{ij}$ are zero. The relations can be chosen with these restrictions because $N$ is conical, see \cite[Proof of Proposition 4.4]{aragoodearl}. Define a hypergraph $H=(H^0, H^1, s, r)$ by $H^0=\{v_j \mid j\in J\}$, $H^1:=\{h_i\mid i\in I\}$, $s(h_i)(v_j)=a_{ij}$ and $r(h_i)(v_j)=b_{ij}$ (here we consider $s(h)_i$ and $r(h_i)$ as maps $H^0\rightarrow \N_0$ associating to each $v_j\in H^0$ its multiplicity in $s(h)$ resp. $r(h)$). It follows from Theorem \ref{thmm} that $V(L_K(H))\cong M(H)\cong N$.
\end{proof}

\section{The graded $V$-monoid}
Throughout this section $\Gamma$ denotes a group with identity $\epsilon$. Let $R$ be a $\Gamma$-graded ring. Recall that a left $R$-module $M$ is called {\it $\Gamma$-graded} if there is a decomposition $M=\bigoplus\limits_{\gamma\in\Gamma}M_\gamma$ such that $R_\alpha M_\gamma\subseteq M_{\alpha\gamma}$ for any $\alpha, \gamma\in \Gamma$. We denote by $R$-$\Gr$ the category of $\Gamma$-graded unital left $R$-modules with morphisms the $R$-module homomorphisms that preserve grading. Furthermore we denote by $R$-$\Gr_{\proj}$ the full subcategory of $R$-$\Gr$ whose objects are the projective objects of $R$-$\Gr$ that are finitely generated as a left $R$-module. If $R$ has graded local units, then
\[V^{\gr}(R)=\{[P] \mid P \in R\text{-}\Gr_{\proj}\},\]
cf. \cite[Subsections 2A,4A]{ara-hazrat-li-sims}. $V^{\gr}(R)$ becomes an abelian monoid by defining $[P]+[Q]=[P\oplus Q]$. 

\subsection{Smash products}

\begin{definition}
Let $R$ be a $\Gamma$-graded ring. The {\it smash product} ring $R \# \Gamma$ is defined as the set of all formal sums $\sum\limits_{\gamma\in\Gamma}r^{(\gamma)}p_{\gamma}$ where $r^{(\gamma)}\in R$ for any $\gamma\in\Gamma$ and the $p_{\gamma}$'s are symbols. Addition is defined component-wise and multiplication is defined by linear extension of the rule $(rp_{\alpha})(sp_{\beta})=rs_{\alpha\beta^{-1}}p_{\beta}$ where $r,s\in R$ and $\alpha,\beta\in \Gamma$.
\end{definition}

Recall that if $R$ is a ring, then $R$-$\Mod$ denotes the category of unital left $R$-modules and $R$-$\Mod_{\proj}$ the full subcategory of $R$-$\Mod$ whose objects are the projective objects of $R$-$\Mod$ that are finitely generated as a left $R$-module. When $R$ has local units, then
$V(R)=\{[P]\mid P\in R\text{-}\Mod_{\proj}\}$.

\begin{proposition}\label{propsmash}
Let $R$ be a $\Gamma$-graded ring with graded local units. Then the categories $R$-$\Gr_{\proj}$ and $R\#\Gamma$-$\Mod_{\proj}$ are isomorphic. It follows that $V^{\gr}(R)\cong V(R\#\Gamma)$.
\end{proposition}
\begin{proof}
Fix a set $E$ of graded local units for $R$. By \cite[Proposition 2.5]{ara-hazrat-li-sims}, there is an isomorphism of categories $\psi:R\text{-}\Gr\rightarrow R\#\Gamma\text{-}\Mod$ which is defined on objects as follows. If $N$ is an object in $R\text{-}\Gr$, then $\psi(N)=N$ as abelian groups. The left-$R\#\Gamma$ action on $\psi(N)$ is defined by $(rp_\alpha)n=rn_\alpha$ for any $r\in R$, $\alpha\in \Gamma$ and $n\in N$. The inverse functor $\phi:R\#\Gamma\text{-}\Mod\rightarrow R\text{-}\Gr$ of $\psi$ is defined on objects as follows. If $M$ is an object in $R\#\Gamma\text{-}\Mod$, then $\phi(M)=M$ as abelian groups. For each $\gamma\in \Gamma$ set $M_\gamma:=\sum\limits_{u\in E} up_\gamma M$. Then $M=\bigoplus\limits_{\gamma\in\Gamma}M_\gamma$. The left $R$-action on $\psi(M)$ is defined by $rm=(rp_\alpha) m$ for any $\alpha,\gamma\in \Gamma$, $r\in R_\gamma$ and $m\in M_\alpha$. One can show that $rm\in M_{\gamma\alpha}$ for any $r\in R_\gamma$ and $m\in M_\alpha$ and hence $\psi(M)$ is a graded left $R$-module. Below we show that $\psi$ restricts to a functor $\psi_{\proj}:R\text{-}\Gr_{\proj}\rightarrow R\#\Gamma\text{-}\Mod_{\proj}$ and that $\phi$ restricts to a functor $\phi_{\proj}:R\#\Gamma\text{-}\Mod_{\proj}\rightarrow R\text{-}\Gr_{\proj}$.\\
Let $P\in R\text{-}\Gr_{\proj}$. Then $\psi(P)$ is a projective object in $R\#\Gamma\text{-}\Mod$ since $\psi$ is an isomorphism of categories. Since $P$ is finitely generated as a left $R$-module, there are $x_1,\dots,x_n\in P$ such that $P=\sum\limits_{i=1}^n Rx_i$. Clearly we can assume that the $x_i$'s are homogeneous (replace each $x_i$ by its homogeneous components). Let $y=\sum\limits_{i=1}^nr_ix_i\in P$ where $r_1, \dots, r_n\in R$. Then $y= \sum\limits_{i=1}^n(r_i p_{\deg(x_i)})x_i$ in $\psi(P)$. This shows that $\psi(P)$ is finitely generated by the $x_i$'s as a left $R\#\Gamma$-module. Thus $\psi(P)\in R\#\Gamma$-$\Mod_{\proj}$.\\
Let now $Q\in R\#\Gamma\text{-}\Mod_{\proj}$. Then $\phi(Q)$ is a projective object in $R\text{-}\Gr$ since $\phi$ is an isomorphism of categories. Since $Q$ is finitely generated as a left $R\#\Gamma$-module, there are $x_1,\dots,x_n\in Q$ such that $Q=\sum\limits_{i=1}^n (R\#\Gamma)x_i$. Clearly we can assume that the $x_i$'s are homogeneous with respect to the grading on $\phi(Q)$. Let $y=\sum\limits_{i=1}^nr_ix_i\in Q$ where $r_1,\dots, r_n\in R\#\Gamma$. For $1\leq i\leq n$ write $r_i=\sum\limits_{\gamma\in\Gamma}r_{i,\gamma}p_\gamma$ where $r_{i,\gamma}\in R$ for any $\gamma\in\Gamma$. Then $y= \sum\limits_{i=1}^n r_{i,\deg(x_i)}x_i$ in $\phi(Q)$ (note that $(r_{i,\gamma}p_{\gamma})x_i=0$ if $\gamma\neq \deg(x_i)$ since $x_i\in Q_{\deg(x_i)}$). This shows that $\phi(Q)$ is finitely generated by the $x_i$'s as a left $R$-module. Thus $\phi(Q)\in R\text{-}\Gr_{\proj}$.\\
By the previous two paragraphs, $\psi$ restricts to a functor $\psi_{\proj}:R\text{-}\Gr_{\proj}\rightarrow R\#\Gamma\text{-}\Mod_{\proj}$ and $\phi$ restricts to a functor $\phi_{\proj}:R\#\Gamma\text{-}\Mod_{\proj}\rightarrow R\text{-}\Gr_{\proj}$. Clearly $\psi_{\proj}\circ\phi_{\proj}=\Id_{R\#\Gamma\text{-}\Mod_{\proj}}$ and $\phi_{\proj}\circ\psi_{\proj}=\Id_{R\text{-}\Gr_{\proj}}$ since $\psi\circ\phi=\Id_{R\#\Gamma\text{-}\Mod}$ and $\phi\circ\psi=\Id_{R\text{-}\Gr}$. Thus $\psi_{\proj}$ is an isomorphism of categories. We leave it to the reader to deduce that the map
\begin{align*}
V^{\gr}(R)&\rightarrow V(R\#\Gamma)\\
[P]&\mapsto [\psi(P)]
\end{align*}
is a monoid isomorphism.
\end{proof}

\subsection{Admissible weight maps}

\begin{definition}\label{defweight}
Let $H$ be a hypergraph and $E$ the directed graph associated to $H$ (see Remark \ref{remhlpa} (b)). Let $w:E^1\rightarrow \Gamma$ be a map such that $w(h_{ij})=w(h_{i1})w(h_{11})^{-1}w(h_{1j})$ for any $h\in H^1$, $1\leq i\leq |s(h)|$ and $1\leq j\leq |r(h)|$. Then $w$ is called an {\it admissible weight map} for $H$.
\end{definition}

\begin{remark}
Let $H$ be a hypergraph. Set $X:=\{h_{1j}, h_{i1}\mid h\in H^1, 1\leq i\leq |s(h)|, 1\leq j\leq |r(h)|\}\subseteq E^1$. Clearly there is a $1$-$1$ correspondence between the set of all maps $X\rightarrow \Gamma$ and the set of all admissible weight maps $E^1\rightarrow \Gamma$ for $H$.
\end{remark}

\begin{lemma}\label{lemweight1}
Let $H$ be a hypergraph and $w:E^1\rightarrow \Gamma$ an admissible weight map for $H$. Then 
\begin{enumerate}[(i)]
\item $w(h_{ik})w(h_{jk})^{-1}=w(h_{il})w(h_{jl})^{-1}$ for any $h\in H^1$, $1\leq i,j\leq |s(h)|$ and $1\leq k,l\leq |r(h)|$ and
\item $w(h_{ki})^{-1}w(h_{kj})=w(h_{li})^{-1}w(h_{lj})$ for any $h\in H^1$, $1\leq i,j\leq |r(h)|$ and $1\leq k,l\leq |s(h)|$. 
\end{enumerate}
\end{lemma}
\begin{proof}
Straightforward computation.
\end{proof}

\begin{lemma}\label{lemweight2}
Let $H$ be a hypergraph and $w:E^1\rightarrow \Gamma$ an admissible weight map for $H$. Then $w$ induces a $\Gamma$-grading on $L_K(H)$ such that $\deg(v)=\epsilon$, $\deg(h_{ij})=w(h_{ij})$ and $\deg(h_{ij}^*)=w(h_{ij})^{-1}$ for any $v\in H^0$ and $h_{ij}\in E^1$.  
\end{lemma}
\begin{proof}
Set $X:=\{v,h_{ij},h_{ij}^*\mid v\in H^0, h\in H^1, 1\leq i\leq |s(h)|,  1\leq j\leq |r(h)|\}$ and let $K\X$ denote the free $K$-algebra generated by $X$. Clearly there is a $\Gamma$-grading on $K\X$ defined by $\deg(v):=\epsilon$, $\deg(h_{ij}):=w(h_{ij})$ and $\deg(h_{ij}^*):=w(h_{ij})^{-1}$ for any $v\in E^0$, $h \in E^{1}$, $1\leq i\leq |s(h)|$ and $1\leq j \leq |r(h)|$. It follows from Lemma \ref{lemweight1} that the relations (i)-(iv) in Definition \ref{defhlpa} are homogeneous. Hence the $\Gamma$-grading on $K\X$ induces a $\Gamma$-grading on $L_K(H)$. 
\end{proof}

\begin{example}\label{exweight1}
Let $H$ be a hypergraph. Set $n:=\sup\{|s(h)|\mid h \in H^{1}\}$. Define a map $w:E^1\rightarrow \Z^n$ by $w(h_{ij})=\alpha_i$ where $\alpha_i$ denotes the element of $\mathbb Z^n$ whose $i$-th component is $1$ and whose other components are $0$. The map $w$ induces the standard grading of $L_K(H)$.
\end{example}

\begin{example}\label{exweight2}
Let $H$ be a hypergraph. Set $m:=\sup\{|s(h)|\mid h \in H^{1}\}$ and $n:=\sup\{|r(h)|\mid h \in H^{1}\}$. Define a map $w:E^1\rightarrow \Z^m\oplus \Z^n$ by $w(h_{ij})=(\alpha_i,\alpha_j)$ where $\alpha_i$ is defined as in the previous example. One checks easily that $w$ is an admissible weight map for $H$ and therefore it induces a $\Z^m\oplus \Z^n$-grading on $L_K(H)$.
\end{example}

\subsection{Covering hypergraphs}

\begin{definition} Let $H$ be a hypergraph and $w:E^1\rightarrow \Gamma$ an admissible weight map for $H$. Define a hypergraph $\overline{H}$ by $\overline{H}^0=\{v_{\gamma}\mid v\in H^0, \gamma\in \Gamma\}$, $\overline{H}^1=\{h_{\gamma}\mid h\in H^1, \gamma\in \Gamma\}$,
\[\overline{s}(h_{\gamma}):=\{(s(h)_1)_{\gamma},(s(h)_2)_{w(h_{21})w(h_{11})^{-1}\gamma },\dots,(s(h)_{|s(h)|})_{ w(h_{|s(h)|,1})w(h_{11})^{-1}\gamma}\}\]
and
\[\overline{r}(h_{\gamma}):=\{(r(h)_1)_{w(h_{11})^{-1}\gamma },\dots,(r(h)_{|r(h)|})_{w(h_{1,|r(h)|})^{-1}\gamma }\}.\]
The hypergraph $\overline{H}$ is called the {\it covering hypergraph} of $H$ defined by $w$.
\end{definition}

\begin{proposition}\label{propcov}
Let $H$ be a hypergraph and $w:E^1\rightarrow \Gamma$ an admissible weight map for $H$. Let $\overline{H}$ be the covering hypergraph of $H$ defined by $w$. Then there is an isomorphism $\phi:L_K(\overline{H})\rightarrow L_K(H)\# \Gamma$ such that \[\phi(v_{\gamma})=vp_{\gamma}, ~\phi((h_{\gamma})_{ij})=h_{ij}p_{w(h_{1j})^{-1}\gamma}\text{ and }\phi((h_{\gamma})^*_{ij})=h_{ij}^*p_{w(h_{i1})w(h_{11})^{-1}\gamma}\]
for any $\gamma\in \Gamma$, $v\in H^0$, $h\in H^1$, $1\leq i\leq |s(h)|$, $1\leq j\leq |r(h)|$.
\end{proposition}
\begin{proof}
Set $a_{v_{\gamma}}:=vp_{\gamma}$, $b_{h_{\gamma},i,j}:=h_{ij}p_{w(h_{1j})^{-1}\gamma}$ and $c_{h_{\gamma},i,j}:=h_{ij}^*p_{w(h_{i1})w(h_{11})^{-1}\gamma}$ for any $v_\gamma\in \overline{H}^0$, $h_\gamma\in \overline{H}^1$, $1\leq i\leq |\overline{s}(h_\gamma)|=|s(h)|$, $1\leq j\leq |\overline{r}(h_\gamma)|=|r(h)|$. In order to show that \[X:=\{a_{v_{\gamma}}, b_{h_{\gamma},i,j},c_{h_{\gamma},i,j}\mid v_\gamma\in \overline{H}^0, h_\gamma\in \overline{H}^1, 1\leq i\leq |\overline{s}(h_\gamma)|, 1\leq j\leq |\overline{r}(h_\gamma)|\}\]
is an $\overline{H}$-family in $L_K(H)\# \Gamma$ one has to show that the relations (i)-(iv) in Remark \ref{remhlpa} (c) are satisfied. We leave (i) and (ii) to the reader and show only (iii) and (iv). \\
Let $h_\gamma\in \overline{H}^1$ and $1\leq i,j\leq |\overline{s}(h_\gamma)|$. Clearly
\begin{align*}
&\sum\limits_{k=1}^{|\overline{r}(h_\gamma)|}b_{h_\gamma,i,k}c_{h_\gamma,j,k}\\
=&\sum\limits_{p=1}^{|\overline{r}(h_\gamma)|}h_{ik}p_{w(h_{1k})^{-1}\gamma}h_{jk}^*p_{w(h_{j1})w(h_{11})^{-1}\gamma}\\
=&\sum\limits_{k=1}^{|\overline{r}(h_\gamma)|}h_{ik}(h_{jk}^*)_{w(h_{1k})^{-1}w(h_{11})w(h_{j1})^{-1}}p_{w(h_{j1})w(h_{11})^{-1}\gamma}\\
=&\sum\limits_{k=1}^{|\overline{r}(h_\gamma)|}h_{ik}(h_{jk}^*)_{w(h_{jk})^{-1}}p_{w(h_{j1})w(h_{11})^{-1}\gamma}\\
=&\sum\limits_{k=1}^{|r(h)|}h_{ik}h_{jk}^*p_{w(h_{j1})w(h_{11})^{-1}\gamma}\\
=&\delta_{ij}s(h)_ip_{w(h_{j1})w(h_{11})^{-1}\gamma}\\
=&\delta_{ij}a_{\overline{s}(h_\gamma)_i}
\end{align*}
and hence (iii) holds. Let now $h_\gamma\in \overline{H}^1$ and $1\leq i,j\leq |\overline{r}(h_\gamma)|$. Clearly
\begin{align*}
&\sum\limits_{k=1}^{|\overline{s}(h_\gamma)|}c_{h_\gamma,k,i}b_{h_\gamma,k,j}\\
=&\sum\limits_{k=1}^{|\overline{s}(h_\gamma)|}h_{ki}^*p_{w(h_{k1})w(h_{11})^{-1}\gamma} h_{kj}p_{w(h_{1j})^{-1}\gamma}\\
=&\sum\limits_{k=1}^{|\overline{s}(h_\gamma)|}h_{ki}^*(h_{kj})_{w(h_{k1})w(h_{11})^{-1}w(h_{1j})}p_{w(h_{1j})^{-1}\gamma}\\
=&\sum\limits_{k=1}^{|\overline{s}(h_\gamma)|}h_{ki}^*(h_{kj})_{w(h_{kj})}p_{w(h_{1j})^{-1}\gamma}\\
=&\sum\limits_{k=1}^{|s(h)|}h_{ki}^*h_{kj}p_{w(h_{1j})^{-1}\gamma}\\
=&\delta_{ij}r(h)_ip_{w(h_{1j})^{-1}\gamma}\\
=&\delta_{ij}a_{\overline{r}(h_\gamma)_i}
\end{align*}
and hence (iv) holds. Thus $X$ is an $\overline{H}$-family in $L_K(H)\# \Gamma$ and therefore there is a unique $K$-algebra homomorphism $\phi:L_K(\overline{H})\rightarrow L_K(H)\# \Gamma$ such that \[\phi(v_{\gamma})=vp_{\gamma}, ~\phi((h_{\gamma})_{ij})=h_{ij}p_{w(h_{1j})^{-1}\gamma}\text{ and }\phi((h_{\gamma})^*_{ij})=h_{ij}^*p_{w(h_{i1})w(h_{11})^{-1}\gamma}\]
for any $\gamma\in \Gamma$, $v\in H^0$, $h\in H^1$, $1\leq i\leq |s(h)|$, $1\leq j\leq |r(h)|$.\\
Clearly the image of $\phi$ contains the set $S:=\{v p_\gamma, h_{ij}p_\gamma, h^*_{ij}p_\gamma\mid v\in H^0, h\in H^1, 1\leq i\leq |s(h)|, 1\leq j\leq |r(h)|\}$. But $S$ generates $L_K(H)\# \Gamma$ as a $K$-algebra and therefore $\phi$ is surjective. It remains to show that $\phi$ is injective. Let $A$ denote the set of all nod-paths $\theta$ for $L_K(\overline{H})$ and $B$ the set of all nod-paths $\eta$ for $L_K(H)$. It is an easy exercise to show that there is an injective map $f:A\rightarrow B\times \Gamma$, $\theta\mapsto (\eta_\theta, \gamma_\theta)$ that has the property that $\phi(\theta)=\eta_\theta p_{\gamma_\theta}$. Now let $x=\sum\limits_{\theta\in A}k_{\theta}\theta\in L_K(\overline{H})$ where almost all coefficients $k_\theta\in K$ are zero. Then $\phi(x)=\sum\limits_{\theta\in A} k_{\theta}\eta_\theta p_{\gamma_\theta}=\sum\limits_{\gamma\in \Gamma}(\sum\limits_{\substack{\theta\in A,\\\gamma_\theta=\gamma}} k_{\theta}\eta_\theta )p_{\gamma}$. Assume now that $\phi(x)=0$. Then $\sum\limits_{\substack{\theta\in A,\\\gamma_\theta=\gamma}} k_{\theta}\eta_\theta=0$ in $L_K(H)$ for any $\gamma\in \Gamma$. But since $f$ is injective, we have $\eta_{\theta_1}\neq \eta_{\theta_2}$ for any $\theta_1, \theta_2\in A$ such that $\theta_1\neq \theta_2$ and $\gamma_{\theta_1}=\gamma_{\theta_2}$. It follows from Theorem \ref{thmbasis} that $k_\theta=0$ for any $\theta\in A$ and hence $x=0$. Thus $\phi$ is injective.
\end{proof}

\subsection{The monoid $V^{\gr}(L_K(H))$}
In this subsection $H$ denotes a hypergraph. We fix an admissible weight map $w:E\rightarrow \Gamma$ for $H$. Recall that $w$ induces a $\Gamma$-grading on $L_K(H)$.

\begin{definition}
We define $M^{\gr}(H)$ as the abelian monoid presented by the generating set $\{v_\gamma\mid v\in H^0, \gamma\in \Gamma\}$ and the relations 
\[\sum\limits_{i=1}^{|s(h)|}(s(h)_i)_{w(h_{i1})w(h_{11})^{-1}\gamma}=\sum\limits_{j=1}^{|r(h)|}(r(h)_j)_{w(h_{1j})^{-1}\gamma }\quad (h\in H^1, \gamma\in\Gamma).\]
\end{definition}

\begin{theorem}\label{thmgrmonoid}
$V^{\gr}(L_K(H))\cong M^{\gr}(H)$.
\end{theorem}
\begin{proof}
Set $L:=L_K(H)$ and $\overline{L}:=L_K(\overline{H})$ where $\overline{H}$ is the covering hypergraph of $H$ defined by $w$. By Propositions \ref{propsmash} and \ref{propcov} we have $V^{\gr}(L)\cong V(L\#\Gamma)\cong V(\overline{L})$. By Theorem \ref{thmm} we have $V(\overline{L})\cong M^{\gr}(H)$. Thus $V^{\gr}(L_K(H))\cong M^{\gr}(H)$.
\end{proof}

\begin{corollary}
Let $(E,C)$ be a finitely separated graph. Then the abelian monoid $V^{\gr}(L_K(E,C))$ defined with respect to the standard $\Z$-grading of $L_K(E,C)$ is presented by the generating set $\{v_\gamma\mid v\in E^0, n\in \Z\}$ and the relations 
\[s(X)_n=\sum\limits_{e\in X}r(e)_{n-1}\quad (X\in C, n\in\Z)\]
where $s(X)$ denotes the common source of the edges in $X$.
\end{corollary}

\begin{corollary}
Let $(E,w)$ be a row-finite vertex-weighted graph. Then the abelian monoid $V^{\gr}(L_K(E,w))$ defined with respect to the standard $\Z^n$-grading of $L_K(E,w)$ (where $n$ is the supremum of the set of all weights) is presented by the generating set $\{v_\gamma\mid v\in E^0, \gamma\in \Z^n\}$ and the relations 
\[\sum\limits_{i=1}^{w(v)}v_{\gamma+\alpha_i-\alpha_1}=\sum\limits_{e\in s^{-1}(v)}r(e)_{\gamma -\alpha_1}\quad (v\in \Reg(E), \gamma\in\Z^n)\]
where $\alpha_i$ denotes the element of $\Z^n$ whose $i$-th component is $1$ and whose other components are zero and $\Reg(E)$ denotes the set of all vertices in $E$ which emit at least one edge.
\end{corollary}

\end{document}